\documentclass[10pt, twoside]{amsart}
\usepackage[utf8]{inputenc}
\usepackage[toc,page]{appendix}
\usepackage{makecell}
\usepackage{stackrel}
\usepackage{tikz-cd}

\usepackage[
backend=biber,
style=alphabetic,
sorting=nyt,
maxbibnames=99
]{biblatex}

\allowdisplaybreaks

\usepackage{imakeidx}
\makeindex

\usepackage{amsmath,amsfonts,amsthm,amssymb,mathrsfs}
\usepackage{hyperref}
\usepackage{graphicx}
\usepackage{multirow}
\usepackage{tikz-cd}
\usepackage{enumerate}
\usepackage{amscd}

\usepackage[hyphenbreaks]{breakurl}

\usepackage[margin=1in]{geometry}

\definecolor{gold}{rgb}{0.88,0.68,0.13}
\definecolor{dgreen}{rgb}{0.0,0.40,0.13}
\definecolor{dblue}{rgb}{0.20,0.20,0.80}

\newcommand{\tbf}[1]{\textbf{#1}}

\newcommand{\te}[1]{\text{#1}}
\newcommand{\cn}{\mathbb{C}}

\newcommand{\tn}{\mathbb{T}}
\newcommand{\zn}{\mathbb{Z}}

\newcommand{\normspace}[2]{\|#2\|_{#1}}

\newcommand{\Prim}{\te{Prim}\,}

\newcommand{\Repn}{\te{Rep}_n\,}
\newcommand{\Repnp}{\te{Rep}_n'\,}

\newcommand{\Irrn}{\te{Irr}_n\,}
\newcommand{\sse}{\subseteq}
\newcommand{\mc}{\mathcal}

\newcommand{\ov}{\overline}

\newcommand{\wh}{\widehat}
\newcommand{\frsub}[1]{{}_{#1}}
\newcommand{\hn}{\mc{H}_n}

\newcommand{\cstarg}{C^*(G)}
\renewcommand{\l}{\left}
\renewcommand{\r}{\right}
\newcommand{\ve}{\varepsilon}

\newcommand{\usf}{\textsf{U}}
\newcommand{\bsf}{\textsf{B}}
\newcommand{\msc}[1]{\mathscr{#1}}
\newcommand{\Tr}{\te{Tr}\,}

\newtheorem{thm}{Theorem}[section]
\newtheorem{prop}[thm]{Proposition}

\newtheorem{cor}[thm]{Corollary}
\newtheorem{lemma}[thm]{Lemma}

\theoremstyle{definition}

\newtheorem{defn}[thm]{Definition}

\newcommand{\dlim}[2]{\displaystyle{\lim_{#1\rightarrow#2}}}
\newcounter{nitem}

\newcounter{alphaitem}

\numberwithin{equation}{section}

\newtheorem{example}[thm]{Example}

\newtheorem{rmk}[thm]{Remark}

\newcommand{\on}{\operatorname}
\newcommand{\mr}{\mathrm}
\newcommand{\B}{\mathbb}

\newcommand{\Z}{\B{Z}}

\newcommand{\ol}{\overline}

\usepackage{bbm}
\newcommand{\mbbm}{\mathbbm}

\newcommand{\T}{\B{T}}

\newcommand{\str}{\texorpdfstring}

\usepackage{ stmaryrd }
\begin{document}
\title{The Topology of the Unitary Dual of Crystallographic Groups}

\author{Frankie Chan}
\address{FC: Department of Mathematics, University of Chicago, Chicago IL 60637, USA}
\email[]{frankiechan@uchicago.edu}

\author{Ellen Weld}
\address{EW: Department of Mathematics and Statistics, Sam Houston State University, Huntsville TX 77341, USA}
\email[]{elw028@shsu.edu}

\date{\today}

\subjclass[2010]{Primary 22D10}
\thanks{\textsc{Department of Mathematics, University of Chicago, Chicago IL 60637, USA}}
\thanks{\textsc{Department of Mathematics and Statistics, Sam Houston State University, Huntsville TX 77341, USA}}

\begin{abstract}
We provide a procedure for generating the irreducible representations of crystallographic groups in any dimension. We also furnish a strategy to investigate the topology of the unitary dual of a crystallographic group using sequences of matrices. All irreducible representations (up to unitary equivalence) of the dimension 3 crystallographic group 90 and some calculations involving sequences of these irreducible representations are included as a proof of concept of this procedure and strategy.
\end{abstract}

\maketitle

%%%%%%%%%%%%%%%%%%%%%%%%%%%%%%%%%%%%%%%%%%%%
%%%%%%%%%%%%%%%%%%%%%%%%%%%%%%%%%%%%%%%%%%%%
\section{Introduction}

Representation theory can offer insight into complicated algebraic structures by converting abstract computations into concrete linear algebra problems. In particular, the topological space of irreducible unitary representations (up to unitary equivalence) of a locally compact group $G$, called the unitary dual, contains rich information about the group and associated algebras. However, two concerns are immediately raised from this line of inquiry. First, generating a complete list of irreducible representations presents its own computational challenge and, second, the unitary dual once assembled is often intractable, frequently failing to be even Hausdorff. 

In this paper, we offer an answer to both of these challenges for the class of crystallographic groups | discrete cocompact groups of isometries of Euclidean space (see \cite{Hil86} for a brief introduction). Crystallographic groups are of interest to areas of mathematics beyond group theory as well as to physics and chemistry. For example, Bieberbach groups (torsion-free crystallographic groups) are exactly the fundamental groups of flat compact Riemannian manifolds (see \cite{Cha86}, \cite{CR03}) and their study has provided useful tools for the investigation of these manifolds. Crystallography, the scientific branch studying molecular and crystalline structure, uses space groups (dimension 3 crystallographic groups) to describe the symmetries of crystals. This field has produced a number of useful resources for investigating space groups such as the Bilbao Crystallographic Server (\url{http://www.cryst.ehu.es}, see \cite{Aetal11}, \cite{Aetal06a}, \cite{Aetal06b}) which provides a variety of tools for analysis of space groups including their irreducible representations. However, the focus of these resources is for physical applications and they do not lend themselves to the analysis of more abstract mathematical objects. For this reason, we remain interested in providing a systematized method of generating irreducible representations from a theoretical mathematics perspective.

This paper uses an abstract group definition that does not explicitly reference Euclidean space, an approach justified by Bieberbach's construction \cite{Bie11}, \cite{Bie12}. Our definition of a {\em crystallographic group of dimension $r$} is a discrete group $G$ fitting into a short sequence of the form
\[1\rightarrow N\rightarrow G\rightarrow D\rightarrow 1\]
where $N\cong \zn^r$ (the {\em lattice}) is maximally abelian in $G$ and $D$ (the {\em point group}) is a finite group (see Section \ref{subsec:crystallography_group_and_Mackey}). 
Because $G$ is finitely generated discrete virtually abelian, all irreducible representations of $G$ are finite dimensional (\cite{Moo72}).

To the first obstacle of producing the irreducible representations of a crystallographic group, we provide code for systematically computing irreducible representations in dimensions 2, 3, and 4 (\cite{ChaWel24}). Our code is written in GAP \cite{GAP4}, a computational group theory program, and uses a package called CrystCat \cite{CrystCat}. This package relies on the work of Brown, B\"{u}low, Neub\"{u}ser, Wondratschek, and Zassenhaus in \cite{BEA78}, a text containing details about crystallographic groups of dimension up to 4. Although the procedure from our code is valid for crystallographic groups of any dimension, the group data for dimensions larger than 4 is not readily available in comprehensive packages. We also note that the runtime can be prohibitively long, with computations most efficiently completed for dimensions 2 and 3.

The procedure to produce these irreducible representations, presented and justified in Section \ref{subsec:the_program}, uses the Mackey Machine (Theorem  \ref{thm:Mackey_machine}). This classic result is extremely powerful and abstractly describes the unitary dual of certain groups as a set. Although this result provides the necessary theory for finding the irreducible representations of the group, it does not provide a practical road map for actually computing them. This difficulty is addressed by using projective representations (see Section \ref{subsec:unitary_projective_representations}) and their deep connections to finite group theory. Due to this reliance on finite groups, our code does require access to a library of finite group representations.

The authors note that a paper by Taylor (\cite{Tay89}) does provide a way to investigate $\wh{G}$ for some crystallographic groups, but does not provide a systematic method. Providing such a system is one motivation for this article.

To the second obstacle, we apply $C^*$-algebraic techniques and focus on sequences of matrices. The unitary dual of a locally compact group is topologized via a canonical bijection with the spectrum of the associated group $C^*$-algebra | a space carrying a natural, albeit complicated, topology (see Section \ref{subsec:spectrum_defn}). The spectra of $C^*$-algebras contains rich information about the algebra but are rarely Hausdorff (in fact, they are not always $T_0$). Fortunately, the $C^*$-algebras arising from crystallographic groups possess a more manageable spectrum. This analysis is conducted in Section \ref{sec:convergence_of_sequences} where we describe and justify a method of using sequences of matrices to uncover topological characteristics in the unitary dual/spectrum of the group $C^*$-algebra. After applying Section \ref{sec:representations_of_crystallography_groups} to produce sequences of irreducible representations $\pi_k$ converging to $\pi$, Theorems \ref{thm:main_result} and \ref{thm:character_theory} then provide the strategy for determining exactly the irreducible components of $\pi$. In Section \ref{sec:group_90}, we exhibit this strategy by examining the dimension 3 crystallographic group 90.

We would also like to note that using sequences to investigate the topology of $\wh{G}$ has been explored in \cite{Bag68} and \cite{Rae82}, but we have specialized it to these particular groups.

%%%%%%%%%%%%%%%%%%%%%%%%%%%%%%%%%%%%%%%%%%%%
%%%%%%%%%%%%%%%%%%%%%%%%%%%%%%%%%%%%%%%%%%%%

\subsection*{Acknowledgments}

The authors want to thank Iason Moutzouris for his careful review and thoughtful comments on previous drafts of this document. They are also grateful to S. Joseph Lippert for his technical help with the code in addition to catching errors in the Appendix.

%%%%%%%%%%%%%%%%%%%%%%%%%%%%%%%%%%%%%%%%%%%%
%%%%%%%%%%%%%%%%%%%%%%%%%%%%%%%%%%%%%%%%%%%%

\section{Preliminaries}\label{sec:prelims}

\subsection{Spectra of Group \str{$C^*$}{C-star}-algebras}\label{subsec:spectrum_defn}
We begin with a brief introduction to $C^*$-algebras and their spectra. Once this theory is established, we will define the analogous group objects and phrase much of our inquiry in terms of those definitions. The primary source for this subsection is the classic text {\em $C^*$-algebras} by Dixmier (\cite{Dix77}).

An associative algebra $A$ over $\cn$ is a {\em $C^*$-algebra} if it is closed under a norm $\|\cdot\|$ and has an involution $*:A\rightarrow A$ such that for all $a,b\in A$,
$$(1)\,\,\,\,\|ab\|\leq \|a\|\,\|b\|,\quad (2)\,\,\,\,\|a^*\|=\|a\|,\quad\te{ and }\quad (3)\,\,\,\,\|aa^*\|=\|a\|^2.$$
Property $(3)$ is called the {\em $C^*$-condition}. For a Hilbert space $\mc{H}$, we let $\bsf(\mc{H})$ denote the set of bounded linear operators on $\mc{H}$. If $A$ is a $C^*$-algebra, a {\em representation} $\pi$ of $A$ on $\mc{H}$ is a linear, multiplicative, involutive map $\pi:A\rightarrow \bsf(\mc{H})$. The {\em dimension} of $\pi$ is the (Hilbert) dimension of $\mc{H}$, which we denote by $\dim \pi$. We say a representation $\pi:A\rightarrow \bsf(\mc{H})$ is {\em nondegenerate} if the set $\{\pi(a)\xi\colon a\in A,\xi\in \mc{H}\}$ is dense in $\mc{H}$. We may always arrange for representations of $C^*$-algebras to be nondegenerate by restricting $\mc{H}$ to a suitable subspace. The celebrated Gelfand-Naimark result states that every $C^*$-algebra has a norm preserving nondegenerate representation and so we may view every $C^*$-algebra ``concretely", that is, as a sub-$*$-algebra of $\bsf(\mc{H})$ for some Hilbert space. 

Fix a $C^*$-algebra $A$. We say two representations, $\pi:A\rightarrow\bsf(\mc{H})$ and $\pi':A\rightarrow \bsf(\mc{H}')$ are {\em equivalent}, denoted $\pi\simeq\pi'$, if there exists a Hilbert space isomorphism $U:\mc{H}\rightarrow\mc{H}'$ such that $U\pi(a)=\pi'(a)U$ for all $a\in A$. If there exists a closed subspace $\mc{K}$ of $\mc{H}$ such that the set $\{\pi(a)\eta\colon a\in A,\eta\in\mc{K}\}\sse \mc{K}$, then we say that $\mc{K}$ is an {\em invariant subspace of $\pi$} and we can define a representation $\pi_{\mc{K}}:A\rightarrow \bsf(\mc{K})$ by $\pi_{\mc{K}}(a)\eta=\pi(a)\eta$ for all $a\in A$, $\eta\in \mc{K}$. We call $\rho\simeq\pi_{\mc{K}}$ a {\em subrepresentation of $\pi$} and write either $\rho\leq \pi$ or $\pi\geq \rho$. We say a representation $\pi:A\rightarrow \bsf(\mc{H})$ is {\em irreducible} if the only closed subspaces $\mc{K}\sse \mc{H}$ which are invariant under $\pi$ are $\{0\}$ and $\mc{H}$.

The {\em kernel} of a representation $\pi$ of $A$ on $\mc{H}$ is the set 
\[\ker\pi=\{a\in A\colon \pi(a)\xi=0\te{ for all }\xi\in\mc{H}\}.\]
A two-sided ideal of $A$ is said to be {\em primitive} if it is the kernel of a non-zero irreducible representation of $A$ on some Hilbert space. The set of all primitive ideals of $A$ is denoted by $\Prim(A)$.

Let $\mathfrak{P}\sse\Prim(A)$ be a collection of primitive ideals of $A$. Define $\mc{I}(\mathfrak{P}):=\bigcap_{I\in\mathfrak{P}}I$ and the {\em closure} of $\mathfrak{P}$ by 
\[\ov{\mathfrak{P}}:=\{J\in\Prim(A)\colon \mc{I}(\mathfrak{P})\sse J\}.\]
These $\ov{\mathfrak{P}}$ generate the {\em Jacobson topology}, which makes $\Prim(A)$ into a $T_0$-space, i.e., for all $I,J\in \Prim(A)$ with $I\neq J$, there exists a neighborhood of $I$ which does not contain $J$ or there exists a neighborhood of $J$ which does not contain $I$. When endowed with the Jacobson topology, we call $\Prim(A)$ the {\em primitive spectrum of $A$}. 

The {\em spectrum of $A$}, denoted by $\wh{A}$, is the set of non-zero irreducible representations under equivalence ($\pi'\in[\pi]\in\wh{A}\iff\pi\simeq\pi'$) endowed with the inverse image of the Jacobson topology under the canonical map
\begin{align*}
\wh{A}&\rightarrow \Prim(A)\\
[\pi]&\mapsto \ker\pi.
\end{align*}
We will call the topology on $\wh{A}$ (the pull-back of the Jacobson topology) the {\em Fell topology}. Let $\wh{A}_n\sse\wh{A}$ be the set of classes of non-zero irreducible representations of dimension $n$ and set ${}_n\wh{A}=\bigcup_{k=1}^{n}\wh{A}_k$. In the Fell topology, ${}_n\wh{A}$ is closed in $\wh{A}$ and $\wh{A}_n$ is open in ${}_n\wh{A}$. 

In general, $\wh{A}$ is extremely complicated though, in special cases, it is $T_0$.

\begin{prop}[\cite{Dix77} 3.1.6 (p.71)]\label{prop:canonical_homeomorphism}
The following are equivalent.
    \begin{itemize}
        \item[(i)] $\wh{A}$ is a $T_0$-space.
        \item[(ii)] Two irreducible representations of $A$ with the same kernel are equivalent.
        \item[(iii)] The canonical map $\wh{A}\rightarrow\Prim(A)$ is a homeomorphism.
    \end{itemize}
\end{prop}

We now turn our attention to the theory required to not only define the (full) group $C^*$-algebra but also the unitary dual. With a single exception, all of the notions discussed above have an analogous notion in the group case.

Let $G$ be a discrete group and let $\usf(\mc{H})$ be the group of unitary operators on a Hilbert space $\mc{H}$. A {\em unitary representation}, or simply {\em representation}, $\pi$ of $G$ in $\mc{H}$ is a group homomorphism $\pi:G\rightarrow \usf(\mc{H})$. The {\em dimension of $\pi$} is $\dim\mc{H}$ and is denoted by $\dim \pi$. There is no notion of degeneracy for unitary representations. We say representations $\pi:G\rightarrow \usf(\mc{H})$ and $\pi':G\rightarrow \usf(\mc{H}')$ are {\em unitarily equivalent}, or {\em equivalent}, if there exists a unitary $U:\mc{H}\rightarrow \mc{H}'$ such that $U\pi(g)=\pi'(g)U$ for all $g\in G$; in which case, we write $\pi\simeq \pi'$.

We say that $\mc{V}$ is a $\pi$-invariant in $\mc{H}$ if $\mc{V}\sse\mc{H}$ and $\pi(g)\xi\in \mc{V}$ for all $g\in G,\xi\in \mc{V}$. $\pi$ is said to be {\em irreducible} if the only $\pi$-invariant closed subspaces of $\mc{H}$ are $\{0\}$ and $\mc{H}$. The set of equivalence classes of all irreducible representations of $G$, denoted by $\wh{G}$, is called the {\em unitary dual} of $G$. Although we should write $[\pi]\in \wh{G}$ to indicate an equivalence class of $\pi$, we will frequently abuse notation and write $\pi\in\wh{G}$ when context is clear. We let $\wh{G}_k\sse \wh{G}$ be classes of irreducible representations of $G$ of dimension $k$ and ${}_n\wh{G}=\bigcup_{k=1}^n\wh{G}_k$. 

When $\mc{V}$ is a $\pi$-invariant closed subspace of $\mc{H}$, we may define $\pi_{\mc{V}}:G\rightarrow\usf (\mc{V})$ by $\pi_{\mc{V}}(g)=\pi(g)|_{\mc{V}}$, which we call a {\em subrepresentation of $\pi$}. We indicate $\rho$ is a subrepresentation of $\pi$ by writing $\rho\leq \pi$ or $\rho\geq \pi$, just as in the $C^*$-case. 
Let $\{\pi_{\lambda}\}_{\lambda\in \Lambda}$ (for $\Lambda\neq 0$) be a collection of representations $\pi_{\lambda}$ of $G$ on $\mc{H}_{\lambda}$ and set $\mc{H}_{\Lambda}=\bigoplus_{\lambda\in\Lambda}\mc{H}_{\lambda}$, the Hilbert space direct sum. Then we may define the {\em direct sum} of $\{\pi_{\lambda}\}_{\lambda\in\Lambda}$ by 
\[\bigoplus_{\lambda\in \Lambda}\pi_{\lambda}(g)((\xi_{\lambda})_{\lambda\in \Lambda})=(\pi_{\lambda}(g)(\xi_{\lambda}))_{\lambda\in\Lambda}\quad \te{ for all }g\in G,(\xi_{\lambda})_{\lambda\in\Lambda}\in\mc{H}_{\Lambda}.\]
If $\pi_{\lambda}=\sigma$ for all $\lambda\in\Lambda$ and the cardinality of $\Lambda$ is $m$, then we will write $\bigoplus_{\lambda\in\Lambda}\pi_{\lambda}=\sigma^{\oplus m}$.
When $\pi$ is finite dimensional, we may always find pairwise inequivalent irreducible representations $\{\sigma_j\}$ such that
\[\pi\simeq\bigoplus_{j=1}^{\ell}\sigma_j^{\oplus m_j}\]
for some $m_j\in\zn_{>0}$. We call each $\sigma_j^{\oplus m_j}$ an {\em isotypic component of $\pi$}.

Given any locally compact group $G$, we define the involutive Banach algebra $L^1(G)$ as the set of absolutely integrable functions with respect to the Haar measure. There exists a one-to-one correspondence between unitary representations of $G$ and nondegenerate representations of $L^1(G)$ (defined as in the case of $C^*$-algebras with obvious modifications) which preserves dimension. The {\em reduced $C^*$-algebra of $G$} is 
\[C_{\lambda}^*(G):=\ov{\lambda_{L^1(G)}(L^1(G))}^{\|\cdot\|_2}\]
where $\lambda_{L^1(G)}$ is the $L^1(G)$ representation associated to $\lambda_G:G\rightarrow \bsf(L^2(G))$ by setting $\lambda_G(s)f(t)=f(s^{-1}t)$ for all $s\in G$. Using the norm on $L^1(G)$ given by
\[\|f\|_u=\sup\{\|\pi(f)\|:\pi\te{ is a *-representation of }L^1(G)\},\]
we define the {\em full group $C^*$-algebra of $G$} to be 
\[C^*(G):=\ov{L^1(G)}^{\|\cdot\|_u}.\]
When $G$ is {\em amenable}, $C^*(G)$ is isomorphic to $C^*_{\lambda}(G)$. For a detailed discussion of this construction and its consequences, see \cite[Ch. VII]{Dav96} or \cite[13.9 (p.303)]{Dix77}.

Once we have constructed $C^*(G)$, we observe that every irreducible representation of $C^*(G)$ is in (dimension preserving) one-to-one correspondence with irreducible unitary representations of $G$ \cite[Ch. VII]{Dav96}. Thus, we may export the topology of $\wh{C^*(G)}$ to $\wh{G}$ via this bijection, which is to say $\wh{C^*(G)}\approx \wh{G}$. In particular, $\wh{C^*(G)}_n\approx \wh{G}_n$ for each $n$.

%%%%%%%%%%%%%%%%%%%%%%%%%%%%%%%%%%%%%%%%%%%%%%%%%%%%%%%%%%%%

\subsection{Concrete versus Abstract Representations}\label{subsec:concrete_versus_abstract}

We want to analyze the spectrum of a $C^*$-algebra using sequences of representations concretely realized as matrices on a shared Hilbert space. We need tools from the ``third definition of the topology of the spectrum" (Section 3.5 in \cite{Dix77}) so that we may draw conclusions about convergence in $\wh{A}$ from convergence of associated matrices.
 
We fix the {\em standard Hilbert space of dimension $n$}, denoted by $\mc{H}_n$, for $n\in\zn_{>0}$ and let $\Repn(A)$ be the set of representations of $A$ on $\hn$ where $\Repnp(A)\sse\Repn(A)$ are those nondegenerate representations. Similarly, we let $\Repn(G)$ be the set of representations of $G$ on $\hn$ (these are automatically nondegenerate). Set $\Irrn(A)\sse\Repn(A)$ to mean the set of non-zero {\em irreducible} representations of $A$ on $\hn$ and define $\Irrn(G)$ analogously in $\Repn(G)$. In particular, when we are using this notation, we are viewing the representation as a family of $n\times n$ matrices | one for each element in either $A$ or $G$, as appropriate. This is in contrast with the abstract view of representations in $\wh{A}$ or $\wh{G}$. 

We topologize $\Repn(A)$ by weak pointwise convergence over $A$; that is, $\pi_k\rightarrow\pi$ for $\pi_k,\pi\in\Repn(A)$ means
\[\langle \pi_k(a)\xi,\eta\rangle_{\hn}\rightarrow \langle \pi(a)\xi,\eta\rangle_{\hn}\quad\te{for any }a\in A,\xi,\eta\in\hn.\]
\cite[3.5.2 (p.80)]{Dix77} shows this is equivalent to strong pointwise convergence over $A$: if $\pi_k\rightarrow\pi$ strongly in $\Repn(A)$, then 
%\begin{equation}\label{eq:equivrepn}
\[\|\pi_k(a)\xi-\pi(a)\xi\|_{\hn}\rightarrow 0\te{ for all }a\in A,\xi\in\hn.\]
%\end{equation}
Of course, we have $\Irrn(A),\Repnp(A)\sse \Repn(A)$ and so we topologize $\Irrn(A)$ and $\Repnp(A)$ as well. We define the same notion of convergence on $\Repn(G)$ and $\Irrn(G)$ (although those matrices are in $\usf(\hn)$ instead of $\bsf(\hn)$). The impact of these sequences on the topology of $\wh{A}$ and $\wh{G}$ is discussed in Section \ref{subsec:spectrum_liminal_Cstar}.

Although our ultimate goal is to analyze $\wh{G}$ (and thus, $\wh{C^*(G)}$ by the discussion above), interrogating sequences in $\usf(\hn)$ has two clear benefits: Hausdorffness and entrywise convergence. Once we establish a meaningful link between sequences in $\usf(\hn)$ and related sequences in $\wh{G}$, we will be able to leverage these qualities of $\usf(\hn)$ to produce useful results. But first, we need to show there is coherence in our notion of convergence of the several spaces of interest. The following proposition justifies passing between convergence on $\usf(\hn)$, $\Repn(G)$, and $\Repnp(C^*(G))$ without comment.

\begin{prop}[\cite{Dix77} 18.1.9 (p.355)]\label{prop:facts_list}\mbox{}
%\begin{enumerate}
       % \item If $A$ is a $C^*$-algebra and  $\pi_{\lambda},\pi\in \Repn(A)$\label{equivrepn}, then
        %$\pi_{\lambda}\rightarrow\pi$ in $\Repn(A)$ if and only if \\$\normspace{\hn}{\pi_{\lambda}(a)\xi-\pi(a)\xi}\rightarrow 0$ for all $a\in A$, $\xi\in \hn$.
        %\item 
        For each $\sigma\in\Repn(G)$, let $\tilde{\sigma}$ be the corresponding element in $\Repn(C^*(G))$. Then following are equivalent (when $G$ is discrete): %\label{prop:group_equivalence}
            \begin{enumerate}
                \item $\pi_{\lambda}\rightarrow\pi$ in $\Repn(G)\sse\usf(\hn)$
                \item $\tilde{\pi}_{\lambda}\rightarrow\tilde{\pi}$ in $\Repn(\cstarg)\sse\bsf(\hn)$
                \item $|\langle \pi_{\lambda}(s)\xi,\eta\rangle_{\hn}-\langle\pi(s)\xi,\eta\rangle_{\hn}|\rightarrow0$ for all $s\in G$, $\xi,\eta\in\hn$
                \item $\normspace{\hn}{\pi_{\lambda}(s)\xi-\pi(s)\xi}\rightarrow0$ for all $s\in G$, $\xi\in\hn$
           % \end{enumerate}
    \end{enumerate}
\end{prop}

We also connect convergence in $\Irrn(A)$ to convergence in $\wh{A}_n$ under special circumstances.

\begin{thm}[\cite{Dix77} 3.5.8 (p.83)]\label{thm:Irrn(A)_to_A_n}
The canonical map $\Irrn(A)$ onto $\wh{A}_n$ given by $\pi\mapsto[\pi]$ is continuous and open.
\end{thm}

%This gives that the topology of $\wh{A}_n$ is the quotient topology under this canonical map.

\begin{rmk}\label{rmk:hausdorff_justification}
Synthesizing this discussion, we conclude that if $\pi_k\rightarrow\pi$ in $\Irrn(G)$, then $[\pi_k]\rightarrow [\pi]$ in $\wh{G}$. Conversely, if $\dim \pi_k=\dim\pi=n$, then $[\pi_k]\rightarrow[\pi]$ in $\wh{G}$ implies $\pi_k\rightarrow\pi$ in $\Irrn(G)$. Because $\usf(\hn)$ is Hausdorff for every cardinal $n$, this implies that if sequences converge in $\Irrn(G)\sse\usf(\hn)$, then the limit is unique. However, $\Irrn(G)$ is not closed in general and a sequence $\{\pi_k\}\sse\Irrn(G)$ can converge to the direct sum of multiple pairwise inequivalent irreducible representations of $G$ (which would mean that, on the level of $\wh{G}$, the sequence possibly converges to distinct elements | see Proposition \ref{prop:sequence_to_subrepresentations}). It is this lack of Hausdorffness which presents the difficulty in studying the unitary dual. One small consolation is that ${}_n\wh{A}$ is closed in $\wh{A}$ in the Fell topology (\cite[3.6.3 (p.85)]{Dix77}) so these distinct elements will always have dimension strictly less than $n$.
\end{rmk}

Before we proceed to discussing abstract group representations, we quickly verify that $\Repn(G)$ is closed in $\usf(\hn)$ when $G$ has a finite presentation.

\begin{lemma}\label{lemma:limit_representations_is_representation}
    Suppose $G$ is a finitely presented group and $\{\pi_k\}\sse\Repn(G)$ and $\pi_k\rightarrow \pi$ in $\usf(\mc{H}_n)$. Then 
    $\pi\in \Repn(G)$.
\end{lemma}

\begin{proof}
    Let $G=\langle g_1,...,g_s\,|\,r_1=\cdots=r_t=e\rangle$ be a presentation of $G$.

    Suppose $r_{\ell}=g_{\sigma(1)}^{f_1}g_{\sigma(2)}^{f_2}\cdots g_{\sigma(m)}^{f_m}$ ($1\leq \ell\leq t$) for $\sigma$ a permutation of the set $\{1,2,...,m\}$ and $f_j\in\zn\slash\{0\}$ ($1\leq j\leq m$). Then, for each $k\in\zn_{>0}$,
    \[
    I_n=\pi_k(r_{\ell})=\pi_k(g_{\sigma(1)})^{f_1}\pi_k(g_{\sigma(2)})^{f_2}\cdots\pi_k(g_{\sigma(m)})^{f_m}\]
    where $I_n$ is the $n\times n$ identity matrix.
    Thus, as $k\rightarrow\infty$,
    \[I_n=\pi(r_{\ell})=\pi(g_{\sigma(1)})^{f_1}\pi(g_{\sigma(2)})^{f_2}\cdots\pi(g_{\sigma(m)})^{f_m}.\]
    Because we have only a finite number of generators and relations, we may arrange for this to hold simultaneously for all $r_{\ell}$, $1\leq \ell\leq t$. Hence, $\pi\in\Repn(G)$.
\end{proof}

%%%%%%%%%%%%%%%%%%%%%%%%%%%%%%%%%%%%%%%%%%%%%%%%%%%%%%%%%%%%

\subsection{Inducing Group Representations}

The source for this subsection is \cite[Ch. 2]{KanTay13}. Let $G$ be a discrete group and $H\leq G$. We will need to extend a representation of $H$ to a representation of $G$ in a ``natural" way. This is achieved by induced representations.

\begin{defn}\label{defn:induced_representation}
    Let $\pi$ be a representation of $H$ on a Hilbert space $\mc{H}_{H}$. We define the {\em induced Hilbert space} by
    \begin{align*}
    \mc{H}_{H}^G:=\{\xi:G\rightarrow\mc{H}_{H}\,|\,\xi(xh)=\pi(h^{-1})\xi(x),&\te{ for }x\in G,h\in H\\
    &\te{ and }\sum_{xH\in G\slash H}\|\xi(x)\|^2<\infty\}.
    \end{align*}
    The {\em induced representation} $\te{ind}_H^G\,\pi$ is the representation of $G$ on $\mc{H}_{H}^G$ defined by
    \[\te{ind}_H^G\,\pi(x)\xi(y):=\xi(x^{-1}y)\quad\te{ for }x,y\in G,\xi\in\mc{H}_{H}^G.\]
\end{defn}

\noindent $\mc{H}_{H}^G$ is isometrically isomorphic to the Hilbert space $\bigoplus_{x H\in G\slash H}\mc{H}_{H}$ via $\xi\mapsto (\xi(xH))_{xH\in G\slash H}$. Consequently, when $\mc{H}_{H}$ is finite dimensional and $[G:H]<\infty$, $\mc{H}_{H}^G$ is finite dimensional.

Combining results in Sections 2.6, 2.7 from \cite{KanTay13}:

\begin{prop}\label{prop:induction_facts}
    Suppose $G$ is a discrete group with $H\leq G$. 
    \begin{enumerate}
        \item Let $\{\pi_{\lambda}\}_{\lambda\in\Lambda}$ be a family of unitary representations of $H$. Then
        \[\textnormal{ind}_H^G\,\l(\bigoplus_{\lambda\in\Lambda}\pi_{\lambda}\r)\simeq\bigoplus_{\lambda\in\Lambda}\textnormal{ind}_H^G\,\pi_{\lambda}.\] \label{prop:induction_direct_sums}
        \item If $\pi$ is a representation of $H$ and $\textnormal{ind}_H^G\,\pi$ is irreducible, then $\pi$ is irreducible. \label{prop:induction_irreducible_representation}
        \item If $K\leq H$ and $\pi$ is a unitary representation of $K$, then \[\textnormal{ind}_K^G\,\pi\simeq\textnormal{ind}_H^G\,\textnormal{ind}_K^H\,\pi.\]\label{prop:induction_in_stages}
    \end{enumerate}
\end{prop}

\subsection{Crystallographic groups and the Mackey Machine}\label{subsec:crystallography_group_and_Mackey}

Crystallographic groups are discrete virtually abelian groups. As such, these groups are particularly well served by the so-called ``Mackey Machine" originally constructed by Mackey (\cite{Mac58}),  a method by which representations may be generated in a near systematic fashion. Yet, even with this powerful tool, it can still be difficult to fully realize all elements of $\wh{G}$. We address this obstacle in Section \ref{subsec:the_program}.

\begin{defn}\label{defn:crystallography_group}
    We say that $G$ is an {\em crystallographic group of dimension $r$} if it is discrete and fits into a short exact sequence of the form
    \[1\rightarrow N\stackrel{i}{\rightarrow} G\stackrel{q}{\rightarrow} D\rightarrow 1\]
    where $N\cong \zn^r$ is maximally abelian in $G$ and $D$ is a finite group.

    $N$ is called the {\em lattice} or {\em translation group} and $D$ the {\em point group}.
\end{defn}

Let $G$ be an $r$-dimensional crystallographic group with lattice $N$ and point group $D$. There is a natural action of $G$ on $N$ defined by $g\cdot n=gng^{-1}$ for all $g\in G$ and $n\in N$. Let $D=G\slash N$ and choose a section $\gamma:D\rightarrow G$ such that $\gamma(1_D)=1_G$. We fix this $\gamma$ for the remainder of the document. Define an action of $D$ on $N$ by $x\cdot n=\gamma(x)\cdot n$. This action is independent of the choice of section as $N$ is abelian. 

It is well known that $\wh{N}\cong \tn^r$ when $N\cong \zn^r$. Hence, we induce an action of $G$ on $\wh{N}$ by
$$g\cdot \chi(n)=\chi(g^{-1}\cdot n)$$
for all $g\in G, \chi\in\wh{N}$, $n\in N$. For each $\chi\in \wh{N}$, we define the {\em stabilizer subgroup associated to $\chi$} by
\[G_{\chi}=\{g\in G:g\cdot \chi=\chi\}\]
and the {\em orbit associated to $\chi$} by 
\[\mc{O}_{\chi}=\{g\cdot\chi:g\in G\}.\]
Of course, $|\mc{O}_{\chi}|={|G\slash G_{\chi}|}$. We note that $N\leq G_{\chi}$ and $|\mc{O}_{\chi}|$ divides $|D|$ for all $\chi\in\wh{N}$.

\begin{rmk}\label{rmk:restriction_induced_reps}
    Suppose $N\leq H\leq K\leq G$ and $\sigma\in\wh{H}$ satisfies ${\sigma|}_N=\chi^{\oplus \dim\sigma}$. Then $$(\on{ind}_H^K\sigma)\big|_N=\bigoplus_{k\in K/H}(k\cdot\chi)^{\oplus\dim\sigma}.$$ In particular, \begin{enumerate}
        \item $(\te{\normalfont{ind}}_{G_{\chi}}^G\sigma)\big|_N=\bigoplus_{g\in G\slash G_{\chi}}(g\cdot \chi)^{\oplus\dim\sigma}$,

        \item if $H\leq G_{\chi}$, then $(\on{ind}_H^{G_{\chi}}\sigma)\big|_N=\chi^{\oplus\dim\sigma}$.
    \end{enumerate}
\end{rmk}

For the rest of the document we use the notation $D_{\bullet}=G_{\bullet}\slash N$ where $N\leq G_{\bullet}\leq G$. In particular, if $\chi\in\wh{N}$, $D_{\chi}=G_{\chi}\slash N$. We now have the notation needed to present the Mackey Machine (our specific phrasing and construction is from \cite[Thm~4.28]{KanTay13}).

Let $\Omega\sse\wh{N}$ intersect each orbit of $G$ exactly once and define $\wh{G}_{\chi}^{(\chi)}$ to be the subset $\sigma\in\wh{G}_{\chi}$ where there exists $m\in\zn_{>0}$ such that
\[\sigma\big|_N=\chi^{\oplus m}.\]

\begin{thm}[Mackey Machine]\label{thm:Mackey_machine}
    $\wh{G}=\l\{\te{\normalfont{ind}}_{G_{\chi}}^G\sigma\,:\,\sigma\in \wh{G}_{\chi}^{(\chi)},\chi\in\Omega\r\}$
\end{thm}

This provides a complete set-theoretic description of $\wh{G}$ and the induction process, once $\sigma\in\wh{G}_{\chi}^{(\chi)}$ is found, is straightforward and systematic. Finding this $\sigma$, however, is more opaque. In the case where $G_{\chi}=G$, we arrive at a situation where, to determine an irreducible representation of $G$, we already need to know the irreducible representations of $G_{\chi}=G$. Naturally, we hope to avoid this type of circular procedure. To know how to systematically approach finding these $\sigma$, we need to discuss so-called {\em unitary projective representations}.

%%%%%%%%%%%%%%%%%%%%%%%%%%%%%%%%%%%%%%%%%%%%
%%%%%%%%%%%%%%%%%%%%%%%%%%%%%%%%%%%%%%%%%%%%

%%%%%%%%%%%%%%%%%%%%%%%%%%%%%%%%%%%%%%%%%%%%
%%%%%%%%%%%%%%%%%%%%%%%%%%%%%%%%%%%%%%%%%%%%

\section{Systematically Generating Irreducible Representations of Crystallographic Groups}\label{sec:representations_of_crystallography_groups}

In this section, we provide our procedure for producing irreducible representations of a crystallographic group. Observe that the arguments used here apply equally well to discrete, finitely generated, virtually abelian groups (that is, abelian-by-finite groups). However, we focus on crystallographic groups. We begin with (unitary) projective representations, present the procedure, and then investigate how this procedure generates sequences of irreducible representations of $G$.

\subsection{Unitary Projective Representations}\label{subsec:unitary_projective_representations}

The next several definitions and results come from Chapters 6 and 7 in \cite{CST22}. As we shall see, projective representations are the bridge between finite groups and crystallographic groups. For this section, let $H$ be a finite group with trivial element $e_H$ and $\mc{A}$ be an abelian group with trivial element $e_{\mc{A}}$.

\begin{defn}
     A {\em 2-cocycle} or simply {\em cocycle} for the pair $(H,\mc{A})$ is a function $\omega:H\times H\rightarrow\mc{A}$ such that
    \begin{align}
    \omega(xy,z)\omega(x,y)&=\omega(x,yz)\omega(y,z)\\
    \omega(e_H,h)&=e_{\mc{A}}=\omega(h,e_H)
    \end{align}
    for all $x,y,z,h\in H$. Let $\msc{C}^2(H,\mc{A})$ be the set of all 2-cocycles relative to $(H,\mc{A})$. Under pointwise multiplication $\msc{C}^2(H,\mc{A})$ has the structure of an abelian group. When $\mc{A}=\T$ is the circle group, we call $\omega$ a {\em unitary (2-)cocycle}.
\end{defn}

We next define an essential 2-cocycle construction: 2-coboundaries generated by a set function.

\begin{defn}\label{defn:2_coboundary}
    We say a set function $\rho:H\rightarrow \mc{A}$ is {\em normalized} if $\rho(e_H)=e_{\mc{A}}$. \textit{The 2-coboundary generated by $\rho$} is a function $\tau_{\rho}\colon H\times H\to\mc{A}$ given by $\tau_{\rho}(x,y)=\rho(xy)[\rho(x)\rho(y)]^{-1}$. A \textit{2-coboundary} is any 2-cocycle of the form $\tau_{\rho}$, for some normalized $\rho$. We let $\msc{B}^2(H,\mc{A})$ denote the set of all 2-coboundaries.

\end{defn}

We omit the proof that $\tau_{\rho}$ is indeed a 2-cocycle. We do observe that $\msc{B}^2(H,\mc{A})$ is a subgroup of $\msc{C}^2(H,\mc{A})$ since $\tau_{\rho_1\rho_2}=\tau_{\rho_1}\cdot\tau_{\rho_2}$ and $\tau_{\rho^{-1}}=(\tau_{\rho})^{-1}$.

\begin{defn}                    
    The quotient $\msc{H}^2(H,\mc{A}):=\msc{C}^2(H,\mc{A})\slash\msc{B}^2(H,\mc{A})$ is the {\em second cohomology group of $H$ with values in $\mc{A}$}. Furthermore, if the cocycles $\omega_1$ and $\omega_2$ are equivalent in $\msc{H}^2(H,\mc{A})$, then we say that $\omega_1$ is \textit{cohomologous to} $\omega_2$, and denote it by $\omega_1\sim\omega_2$. In the case $\omega\sim\mbbm{1}$ (the trivial element of $\msc{H}^2(H,\mc{A})$), we say $\omega$ is \textit{cohomologically trivial}.
\end{defn}

Fix a character $\chi\in\wh{N}$, whose stabilizer modulo $N$ is $D_{\chi}\leq D$. Define the factor set $\nu\colon D\times D\to N$, where for each $h,k\in D$,
$$\nu(h,k)=\gamma(hk)^{-1}\gamma(h)\gamma(k).$$
We define {\em the unitary 2-cocycle corresponding to $\chi$}, $\omega_{\chi}\colon D_{\chi}\times D_{\chi}\to\T$, by
$$\omega_{\chi}(h,k)=\ol{\chi}(\nu(h,k)).$$ 
We note that while the function definition and domain of $\omega_{\chi}$ may be extended to all of $D$, this extended $\omega_{\chi}$ need not remain a 2-cocycle. The cohomology class of $\omega_{\chi}$, which is independent of the choice of section $\gamma$, is called the {\em Mackey obstruction of $\chi$} (see \cite[Rmk 4.54 (p.181)]{KanTay13}. Lifting the domain from $D_{\chi}$ to $G_{\chi}$, the \textit{inflation} of $\omega_{\chi}$ will also be denoted $\omega_{\chi}\colon G_{\chi}\times G_{\chi}\to\T$ and is given by $\omega_{\chi}(x,y)=\ol{\chi}(\nu(q(x),q(y)))$ where $q:G\rightarrow D$ is the quotient map. 

\begin{defn}[\cite{CST22}, Defn 7.8]
    Let $\omega\in\msc{C}^2(H,\T)$. We say the function $\pi^P:H\rightarrow \usf(\mc{V})$ is a {\em (projective) $\omega$-representation} if 
    $$\pi^P(xy)=\omega(x,y)\pi^P(x)\pi^P(y),$$
    for all $x,y\in H$.
\end{defn}

Since $\omega$ is assumed to be normalized, all projective $\omega$-representations satisfy $\pi^P(e_H)=I_{\mc{V}}$. 

The definitions from Section \ref{subsec:spectrum_defn} introduced for ordinary representations have analogs for projective representations, such as irreducible projective representations and unitary equivalence. Denote the set of equivalence classes of irreducible projective $\omega$-representations of $H$ by $\wh{H}^{\omega}$. In particular, let $\wh{D}_{\chi}^{\eta}$ be the set of equivalence classes of irreducible $\eta$-representations of $D_{\chi}$.

\begin{prop}[\cite{CST22}, Remark 7.9(3)]
    Let $\Phi_1^P$ and $\Phi_2^P$ be an $\omega_1$-representation and an $\omega_2$-representation of $H$, respectively. There exists a (normalized) function $\lambda\colon H\to\T$ such that $\Phi_1^P(h)=\lambda(h)\Phi_2^P(h)$ for all $h\in H$ if and only if $\omega_1$ and $\omega_2$ are cohomologous. Moreover, if $\omega_1=\omega_2\cdot\tau_{\lambda}$, the map $\wh{H}^{\omega_2}\to\wh{H}^{\omega_1}$ given by $\Phi_2^P\mapsto\Phi_2^P\cdot\lambda$ is a bijection (and preserves equivalence of representations).
\end{prop}

% Move this somewhere else
% \begin{prop}[\cite{CST22} Proposition7.4 (p.191)]
%     The equivalence class of $\omega_{\chi}$ is trivial in $\msc{H}^2(D_{\chi},\tn)$ if and only if there exists $\sigma\in\wh{G}_{\chi}^{(\chi)}$ such that $\dim \sigma=1$.
% \end{prop}

% \begin{thm}[\cite{KanTay13} Thm 4.53 (p.181)]
%     There exists a dimension preserving bijection between $\wh{D}_{\chi}^{\ov{\omega_{\chi}}}$ and $\wh{G}_{\chi}^{(\chi)}$.
% \end{thm}

% Appealing to $\wh{D}_{\chi}^{\ov{\omega_{\chi}}}$ has certain benefits as $D_{\chi}$ is a finite group. For example, part (3) of Corollary7.15 (p.199) in \cite{CST22} implies
% \begin{equation}\label{dimensionargument1}
% |D_{\chi}|=\sum_{i=1}^k(\dim\sigma_i)^2
% \end{equation}
% where $|\wh{D}_{\chi}^{\ov{\omega_{\chi}}}|=k$ and $\{\sigma_i\}_{i=1}^k$ are pairwise inequivalent $\ov{\omega_{\chi}}$-representations of $D_{\chi}$. 

Given a character $\chi\colon N\to\T$ via $\chi^*\colon G_{\chi}\to\T$, we extend its domain to $G_{\chi}$ by writing $\chi^*(x):=\chi(\gamma(q(x))^{-1}x)$.

\begin{prop}[\cite{KanTay13}, Lemma 4.48]\label{prop:chi_star}
    The function $\chi^*\colon G_{\chi}\to\T$ is a 1-dimensional $\ol{\omega_{\chi}}$-representation.
\end{prop}

\begin{defn}[\cite{CST22}]\mbox{}
    \begin{enumerate}
        \item A cocycle $\omega\in\msc{C}^2(H,\mc{A})$ is \textit{equalized}, if $\omega(x,x^{-1})=e_{\mc{A}}$, for all $x\in H$.
        \item A unitary cocycle $\omega\in\msc{C}^2(H,\T)$ is \textit{finitized}, if the image of $\omega$ is contained in some $\Z_n$ ($n\in\zn_{>0}$). Here, we are treating $\Z_n\leq\T$ as the subgroup of the $n$-th roots of unity.
    \end{enumerate}
\end{defn}

The primary purpose of the equalization and finitization functions are to prove Proposition \ref{prop:equalized_finitized_bijection}. The following procedure is justified in Proposition 6.15 and Theorem 6.19 of \cite{CST22}.

\begin{defn}
    Let $H$ be a finite group of order $n$ and fix a unitary 2-cocycle $\omega\in\msc{C}^2(H,\T)$. With respect to $\omega$, we define the following normalized functions.
    \begin{enumerate}
        \item The \textit{equalization function} $\mr{eq}\colon H\to\T$ is given by $\mr{eq}(h)=\sqrt{\omega(h,h^{-1})}$, where we choose the branch of square root so that the complex argument $\mr{Arg}(\mr{eq}(h))\in[0,\pi)$.

        \item Let $\omega$ be equalized. We define the \textit{finitization function} $\mr{fin}\colon H\to\T$ in a series of steps. \begin{enumerate}
            \item Define $\rho\colon H\to\T$ via $\rho(k)=\prod_{h\in H}\omega(h,k)$. It is a consequence that $\rho(k)=\rho(k^{-1})$ for every $k\in H$.
            
            \item We define the \textit{finitization function} $\mr{fin}\colon H\to\T$ piecewise as follows. These three possibilities are mutually exclusive and exhaustive. \begin{enumerate}
                \item If $h=h^{-1}$, define $\mr{fin}(h)=1$.
                
                \item If $h\neq h^{-1}$ and $\rho(h)=\rho(h^{-1})=-1$, arbitrarily choose an ordering for the pair $(h,h^{-1})$ and map $\mr{fin}(h)=e^{i\pi/n}$ and $\mr{fin}(h^{-1})=e^{-i\pi/n}$. This chosen order will not matter for our needs.
            
                \item Lastly, if $\rho(h)=-1$, define $\mr{fin}(h)=\left(\rho(h)\right)^{1/n}$ where $\mr{Arg}\left(\rho(h)\right)^{1/n}\in(-\pi/n,\pi/n]$.
            \end{enumerate}
        \end{enumerate} 
    \end{enumerate}
\end{defn}

\begin{prop}[\cite{CST22}, Thm 6.19]\label{prop:equalized_finitized_omega}
    Suppose $H$ is a finite group of order $n$. Every unitary 2-cocycle $\omega\in\msc{C}^2(H,\T)$ is cohomologous to a 2-cocycle $\omega_{\mr{fin}}$ that is simultaneously equalized and finitized, with codomain $\Z_n$. In particular, $\msc{H}^2(H,\T)\cong\msc{H}^2(H,\Z_n)$ is a finite group.
\end{prop}

The details can be found in \cite[Thm 6.19]{CST22}, where the 2-coboundaries $\tau_{\mr{eq}}$ and $\tau_{\mr{fin}}$ are used. It can be verified that for any unitary 2-cocycle $\omega$, $\omega_{\mr{eq}}:=\omega\cdot\tau_{\mr{eq}}$ is an equalized cocycle. It can also be shown that $\omega_{\mr{fin}}:=\omega_{\mr{eq}}\cdot\tau_{\mr{fin}}^{-1}$ is an equalized and finitized cocycle.

\begin{prop}\label{prop:equalized_finitized_bijection}
    There is a bijection $\wh{D_{\chi}}^{\omega_{\mr{fin}}}\xrightarrow{\sim}\wh{D_{\chi}}^{\omega_{\chi}}$ given by $\Phi^P\mapsto\Phi^P\cdot\mr{eq}^{-1}\cdot\mr{fin}$.
\end{prop}

\begin{proof}
    The proof follows from \cite[Rmk 7.9(3)]{CST22}.
\end{proof}

\begin{thm}[\cite{KanTay13}, Thm 4.53]\label{thm:projective_to_regular_bijection}
    There is a bijection $\wh{D_{\chi}}^{\omega_{\chi}}\xrightarrow{\sim}\wh{G}_{\chi}^{(\chi)}$ given by $\pi^P\mapsto\sigma$, where $\sigma(x):=\chi^*(x)\pi^P(q(x))$.
\end{thm}

\begin{proof}
    Given an irreducible $\omega_{\chi}$-representation $\pi^P$ of $D_{\chi}$, the map $\pi^P\circ q\colon G_{\chi}\to\on{U}(\mc{V})$ is an $\omega_{\chi}$-representation of $G_{\chi}$, and by Proposition \ref{prop:chi_star}, $\chi^*\colon G_{\chi}\to\T$ is an $\ol{\omega_{\chi}}$-representation. Hence, the product $\sigma=\chi^*\cdot(\pi^P\circ q)$ is an irreducible projective representation of $G_{\chi}$ with 2-cocycle $\omega_{\chi}\cdot\ol{\omega_{\chi}}\equiv\mbbm{1}$, i.e., a bonafide group homomorphism. The proof of the remaining statements is covered in \cite[Sec 4.7]{KanTay13}.
\end{proof}

\begin{rmk}
We emphasize that the bijections in Proposition \ref{prop:equalized_finitized_bijection} and Theorem \ref{thm:projective_to_regular_bijection} {\em preserve} irreducibility.
\end{rmk}

It remains to find a systematic way to produce projective irreducible representations in $\wh{D_{\chi}}^{\omega_{\chi}}$. We are able to answer this question with a detour into (ordinary) group representations of finite groups. Moreover, to find $\wh{H}^{\omega}$ with $H$ a finite group of order $n$ and $\omega\in\msc{C}^2(H,\Z_{n})$ a finitized unitary 2-cocycle, the group extension of $H$ with respect to $\omega$ will prove useful.

\begin{defn}\label{defn:group_extension}
    Given a 2-cocycle $\omega\colon H\times H\to\mc{A}$, define the {\em group extension $G(\mc{A},H,\omega)$ of $H$ with kernel $\mc{A}$} such that
    \begin{enumerate}
        \item as sets, $G(\mc{A},H,\omega)=\mc{A}\times H$;
        \item the group operation $*$ on $G(\mc{A},H,\omega)$ is given by $(a_1,h_1)*(a_2,h_2):=(a_1a_2\omega(h_1,h_2)^{-1},h_1h_2)$;
        \item the element $(e_{\mc{A}},e_H)$ is the identity in $G(\mc{A},H,\omega)$; and
        \item inverses are given by $(a,h)^{-1}=(a^{-1}\omega(h,h^{-1}),h^{-1})$, for all $a\in\mc{A}$ and $h\in H$.
    \end{enumerate}
\end{defn}

This definition indeed defines a group. Furthermore, we have the following short exact sequence $$0\to\mc{A}\xrightarrow{i} G(\mc{A},H,\omega)\xrightarrow{s} H\to0,$$ where $i(a)=(a,e_H)$ and $s(a,h)=h$ and $i(\mc{A})$ is a central subgroup of $G(\mc{A},H,\omega)$.

Set $d:=|D_{\chi}|$. Each $\omega_{\chi}\in\msc{C}^2(D_{\chi},\T)$ is cohomologous to some equalized and finitized $\omega_{\mr{fin}}\in\msc{C}^2(D_{\chi},\Z_d)$ by Proposition \ref{prop:equalized_finitized_omega}. The group extension $G(\Z_d,D_{\chi},\omega_{\mr{fin}})$ is a finite group of order $d^2$. 

\begin{defn}
    The set $\wh{G}^*(\Z_d,D_{\chi},\omega_{\mr{fin}})$ consists of all of the (equivalence classes of) irreducible representations $\Theta\colon G(\Z_d,D_{\chi},\omega_{\mr{fin}})\to\usf(\mc{V})$ which satisfy $\Theta(\lambda,1_D)=\lambda I_{\mc{V}}$.
\end{defn}

\begin{thm}[\cite{CST22}, Prop 7.12]
    There is a bijection $\wh{G}^*(\Z_d,D_{\chi},\omega_{\mr{fin}})\xrightarrow{\sim}\wh{D_{\chi}}^{\omega_{\mr{fin}}}$ given by $\Theta\mapsto\Phi^P$, where $\Phi^P(d):=\Theta(1,d)$.
\end{thm}

Finally, we characterize the relationship between $\wh{G}_{\chi}^{(\chi)}$ and projective representations.

\begin{thm}\label{thm:long_bijection}
    Fix a character $\chi\in\wh{N}$ with stabilizer $G_{\chi}$ where the finite quotient $D_{\chi}=G_{\chi}\slash N$ is of order $d$. Take $\chi^*\colon G_{\chi}\to\T$ to be the extension of $\chi$. There is a bijection $$\begin{array}{clcrccc}
    \wh{G}^*(\Z_d,D_{\chi},\omega_{\mr{fin}}) & \xrightarrow{\sim} & \wh{D_{\chi}}^{\omega_{\mr{fin}}} & \xrightarrow{\sim} & \wh{D_{\chi}}^{\omega_{\chi}} & \xrightarrow{\sim} & \wh{G}_{\chi}^{(\chi)}\\
    \Theta & \mapsto & \Phi^P & \mapsto & \Phi^P\cdot\mr{eq}^{-1}\cdot\mr{fin} & \mapsto & \chi^*\cdot((\Phi^P\cdot\mr{eq}^{-1}\cdot\mr{fin})\circ q),
\end{array}$$ where $\Phi^P(d):=\Theta(1,d)$. The dimensions of the images of each representation are preserved.
\end{thm}

We observe that our desired application of Theorem \ref{thm:long_bijection} requires a complete list of irreducible unitary representations of the finite group $G(\Z_d,D_{\chi},\omega_{\mr{fin}}$). We use the computational group theory system \cite{GAP4} to produce the set $\wh{G}^*(\Z_d,D_{\chi},\omega_{\mr{fin}})$.

%This seems like a good alternative source on projective reps: %http://maths.nju.edu.cn/~ccheng/Writings/preps-notes.pdf

%%%%%%%%%%%%%%%%%%%%%%%%%%%%%%%%%%%%%%%%%%%%
\subsection{Systematically Generating Irreducible Representations}\label{subsec:the_program}

We now justify the mathematical underpinnings of the GAP code found on \cite{ChaWel24}. Expanding this subsection to cover the class of all finitely presented discrete virtually abelian groups is the focus of a future project. Moreover, we are working to improve the code to further systematize the process.

Fix a crystallographic group $G$ with lattice $N$ and point group $D$. Define a factor set $\nu:D\times D\rightarrow N$ based on the fixed normalized section $\gamma:D\rightarrow G$.

\tbf{The Procedure:} 
\begin{enumerate}
\item Select a concrete $\chi\in\wh{N}$ and determine $G_{\chi}$, $D_{\chi}=G_{\chi}\slash N$. Here, concrete means $\chi\in\wh{N}$ is chosen with specific numerical values. See the discussion in Section \ref{subsec:group_90}.
\item Define the unitary 2-cocycle corresponding to $\chi$: $\omega_{\chi}\colon D_{\chi}\times D_{\chi}\rightarrow\tn$ given by $\omega_{\chi}:=\ov{\chi}\circ \nu$.
\item Replace $\omega_{\chi}$ by a cohomologous equalized and finitized unitary 2-cocycle $\omega_{\mr{fin}}$ following the method in Proposition \ref{prop:equalized_finitized_omega}.
\item Create the group extension $G(D_{\chi},\Z_d,\omega_{\mr{fin}})$, a finite group, via Definition \ref{defn:group_extension}.
\item Find the irreducible representations of $G(D_{\chi},\Z_d,\omega_{\mr{fin}})$ using GAP (or similar library of finite group representations).
\item Use Theorem \ref{thm:long_bijection} to produce $\wh{G}_{\chi}^{(\chi)}$.
\item Induce from $G_{\chi}$ to $G$ as in Definition \ref{defn:induced_representation}.
\end{enumerate}

After considering all $[\chi]\in\wh{N}\slash D$, we see this procedure exhausts $\wh{G}$ by the Mackey Machine. 

\subsection{Constructing Convergent Sequences in \str{$\wh{D}_*$}{D-hat-star}}\label{subsec:constructing_convergent_sequences}
Our next step is to lay the groundwork for generating convergent sequences in $\wh{G}$, which will be leveraged in Sections \ref{subsec:main_top_results} and \ref{subsec:unitary_dual_via_sequences}. We continue to see the importance of the structure of projective representations.

\begin{lemma}\label{lem:set_fcts_are_projReps}
    Let $H$ be a finite group and $\lambda\colon H\to\T$ be any normalized set function. Then $\lambda$ is a $\tau_{\lambda}$-representation of $H$ (see Definition \ref{defn:2_coboundary}).
\end{lemma}

\begin{proof}
    Define $\eta\colon H\times H\to\T$ via $\lambda(hk)=\eta(h,k)\lambda(h)\lambda(k)$. Then by definition, $\eta=\tau_{\lambda}$.
\end{proof}

%Let $\on{Fun}^*(H,\T)$ be the set of normalized functions $\tau\colon H\to\T$. 

%\begin{prop}
%    The map $\tau\colon\on{Fun}^*(H,\T)\twoheadrightarrow\msc{B}^2(H,\T)$ given by $\lambda\mapsto\tau_{\lambda}$ is an $\alpha$-to-1 surjective group homomorphism, where $\alpha:=|\on{Hom}(H,\T)|<\infty$.
%\end{prop}

% \begin{proof}
%     By the definition of $\msc{B}^2(H,\T)$, we have that $\tau$ is a surjective group homomorphism. It remains to show that $\on{Ker}\tau=\on{Hom}(H,\T)$. Observe $\on{Hom}(H,\T)=\wh{H}_{\mr{1D}}^{\mbbm{1}}$, so by Lemma \ref{lem:set_fcts_are_projReps}, we have $\wh{H}_{\mr{1_D}}^{\mbbm{1}}\subseteq\on{Ker}\tau$. Now, suppose $\eta\in\on{Ker}\tau$. Again by Lemma \ref{lem:set_fcts_are_projReps}, we have that $\eta\colon H\to\T$ is a $\tau_{\eta}\equiv\mbbm{1}$-representation, i.e., $\eta\in\on{Hom}(H,\T)$. Lastly, the set $\on{Hom}(H,\T)$ is finite since $\on{Hom}(H,\T)\cong\on{Hom}(H,\Z_{|H|})$.
% \end{proof}

% Given an irreducible projective representation $\pi$ of $H$ we can produce additional irreducible representations by composing with a homomorphism -- \tbf{\textcolor{blue}{Not sure if this is what we mean to say}}

\begin{defn}
    Fix a subgroup $D_*\leq D$. We say two characters $\chi,\chi'\in\wh{N}$ are \emph{cohomologous} (over $D_*$), if $\omega_{\chi},\omega_{\chi'}\in\msc{C}^2(D_*,\T)$ and $\omega_{\chi}\sim\omega_{\chi'}$.
\end{defn}

Let $\chi_k\to\chi$. By reducing to a subsequence, we can assume that each stabilizer $D_*=D_{\chi_k}$, for some fixed $D_*\leq D$. Each of the $\omega_{\chi_k}$ are 2-cocycles over $D_*$. That $\omega_{\chi}$ is {\em also} a 2-cocycle over $D_*$ is justified in Proposition \ref{prop:conv_chars}. 

\begin{prop}\label{prop:eventually_trivial}
    Suppose that $\chi_k\to\mbbm{1}$. Then eventually, the sequence $\chi_k$ is cohomologically trivial over $\msc{C}^2(D_*,\T)$.
\end{prop}

\begin{proof}
    We may reduce to a subsequence so $\chi_k\to\mbbm{1}$ such that each $\omega_k:=\ol{\chi_k}\circ\nu$ belongs to the same cohomology class of the finite set $\msc{H}^2(D_*,\T)$. Hence, we have a sequence of (normalized) functions $\lambda_k\colon D_*\to\T$ such that $\omega_k=\omega_0\cdot\tau_{\lambda_k}$. We may insist that $\lambda_k$ converges pointwise to some normalized function $\lambda\colon D_*\to\T$ by reducing to a further subsequence. Since $\omega_0\cdot\tau_{\lambda_k}=\omega_k\to\mbbm{1}$, it follows that $\tau_{\lambda_k}\to\ol{\omega_0}$. Consequently, $\omega_0=\tau_{\ol{\lambda}}$, and each successive $\omega_k$ must also be cohomologically trivial.
\end{proof}

\begin{cor}\label{cor:cohomologous_sequence}
    Suppose that $\chi_k\to\chi$. Then the sequence $\chi_k$ is eventually cohomologous to $\chi$ over $\msc{C}^2(D_*,\T)$.
\end{cor}

\begin{proof}
    Observe that $\chi_k\cdot\ol{\chi}\to\mbbm{1}$ and then apply Proposition \ref{prop:eventually_trivial}.
\end{proof}

Take a convergent sequence $\chi_k\to\chi\in\wh{N}$, subsequentially reduced to have a common stabilizer subgroup $D_*\leq D_{\chi}$, and such that each $\chi_k$ are cohomologous to $\chi$ over $D_*$ (which we may arrange by Proposition \ref{prop:conv_chars}). We can therefore find a sequence of normalized functions $\lambda_k\colon D_*\to\T$ such that $\omega_k:=\ol{\chi_k}\circ\nu=\omega_0\cdot\tau_{\lambda_k}$, for each $k\geq1$. We may also assume that $\lambda_k\to\lambda$ (pointwise on $D_*$) and $\omega=\omega_0\cdot\tau_{\lambda}$. By Remark 7.9(3) in \cite{CST22}, we have bijections given by 
$$\begin{array}{clcrc}
    \wh{D_*}^{\omega_k} & \xleftarrow{\sim} & \wh{D_*}^{\omega_0} & \xrightarrow{\sim} & \wh{D_*}^{\omega}\\
    \Phi^P\cdot\lambda_k & \mapsfrom & \Phi^P & \mapsto & \Phi^P\cdot\lambda.
\end{array}$$

Let $\Phi^P$ be an $\omega_0$-representation of dimension $n$ and fix the standard Hilbert space $\hn$. We can then produce a sequence 
\begin{equation}\label{eq:sequence}
\Phi^P,\ \Phi^P\cdot\lambda_1,\ \Phi^P\cdot\lambda_2,\dots\to\Phi^P\cdot\lambda
\end{equation}
of irreducible projective representations of $D_*$ in $\hn$ corresponding to each $\omega_k$ which converges to the $\omega$-representation $\Phi^P\cdot\lambda$ (this $\omega$-representation is {\em also} irreducible over $D_*$). 

Since $|\wh{D_*}^{\omega_k}|=|\wh{D_*}^{\omega}|<\infty$, this is the only way these convergent sequences of irreducible projective representations can arise. That is, given a convergent sequence $\pi_k^P\to\pi^P$ over $D_*$ in $\hn$ where each $\pi_k^P$ and $\pi^P$ are irreducible projective representations under cohomologous 2-cocycles, we can reduce to a subsequence and set $\pi_0^P=\Phi^P$, $\pi_k^P=\Phi^P\cdot\lambda_k$, and $\pi^P=\Phi^P\cdot\lambda$ so that $\Phi^P,\ \Phi^P\cdot\lambda_1,\ \Phi^P\cdot\lambda_2,\dots\to\Phi^P\cdot\lambda$.  Clearly, if $\chi_k\to\chi$, then $\chi_k^*\to\chi^*$ over $G_*$. By Theorem 4.53 in \cite{KanTay13}, we have the bijection 
\begin{equation}\label{eq:sequence2}
\begin{array}{clc}
    \wh{D_*}^{\omega_k} & \xrightarrow{\sim} & \wh{G_{*}}^{(\chi_k)}\\
   \pi_k^P & \mapsto & \chi_k^*\cdot\pi_k^P,
\end{array}
\end{equation}
where in the image, we are using $\pi_k^P$ to mean $\pi_k^P\circ q$. This gives us the sequence of irreducible ordinary representations in $\wh{G_*}^{(\chi_k)}$, $$\chi_0^*\cdot\Phi^P,\ \chi_1^*\cdot\Phi^P\cdot\lambda_1,\ \chi_2^*\cdot\Phi^P\cdot\lambda_2,\dots\to\chi^*\cdot\Phi^P\cdot\lambda,$$ converging in $\usf(\hn)$ to an irreducible ordinary representation $\chi^*\cdot\Phi^P\cdot\lambda$ of $G_*$ in $\hn$ such that ${(\chi^*\cdot\Phi^P\cdot\lambda)|}_N=\chi^{\oplus n}$.

\begin{defn}\mbox{}
    \begin{enumerate}
        \item $\on{R}^{\omega}(D_\chi)$ (resp. $\wh{D_{\chi}}^{\omega}$) is the set of equivalence classes of (resp. irreducible) projective $\omega$-representations of $D_{\chi}$. 
        \item $\on{R}_{\chi}(G_\chi)$ (resp. $\wh{G}_{\chi}^{(\chi)}$) is the set of equivalence classes of (resp. irreducible) ordinary representations of $G_{\chi}$, which restrict to $N$ as a direct sum of $\chi$.
    \end{enumerate}
\end{defn}

\begin{prop}[\cite{KanTay13}, Thm 4.53, modified]\label{prop:projective_lift}
    The map $\on{R}^{\omega}(D_{\chi})\xrightarrow{\sim}\on{R}_{\chi}(G_{\chi})$ given by $\pi^P\mapsto\chi^*\cdot(\pi^P\circ q)$ is a bijection and preserves direct sums. Furthermore, this map restricts to a bijection $\wh{D_{\chi}}^{\omega}\xrightarrow{\sim}\wh{G}_{\chi}^{(\chi)}$.
\end{prop}

%%%%%%%%%%%%%%%%%%%%%%%%%%%%%%%%%%%%%%%%%%%%
%%%%%%%%%%%%%%%%%%%%%%%%%%%%%%%%%%%%%%%%%%%%
\section{Sequences in \str{$\wh{G}$}{G-hat}}\label{sec:convergence_of_sequences}

\subsection{Spectrum of a Liminal \str{$C^*$}{C-star}-algebra}\label{subsec:spectrum_liminal_Cstar}

Our $C^*$-algebras of interest are separable and liminal. We say $A$ is a {\em liminal} $C^*$-algebra if, for all irreducible representations $\pi$ of $A$, $\pi(x)$ is compact for all $x\in A$. When $A$ is liminal, every point in $\wh{A}$ is closed \cite[4.4.1 (p.102)]{Dix77}. 

This leads to an important corollary.

\begin{cor}\label{cor:closed_points_app}
    Let $A$ be a liminal $C^*$-algebra with $\rho,\sigma$ irreducible representations. If $\ker\rho\sse\ker\sigma$, then $\rho\simeq\sigma$.
\end{cor}

%\begin{proof}
 %   If $\ker\rho\sse\ker\sigma$, then $\ker\sigma\in\ov{\{\ker\rho\}}=\{\ker\rho\}$. Thus, $\ker\sigma=\ker\rho\iff \sigma\simeq\rho$ as desired.
%\end{proof}

Moreover, liminal $C^*$-algebras are automatically Type I, where we say a $C^*$-algebra $A$ is {\em Type I} if and only if for any two irreducible representations $\pi_1,\pi_2$ of $A$ with the same kernel are unitary equivalent. According to Proposition \ref{prop:canonical_homeomorphism}, we then have a homeomorphism between $\Prim(A)$ and $\wh{A}$ (see \cite[Thm IV.15.7 (p.339)]{Bla10} for details on these claims). We now begin with a few easy results which provide us with basic tools for using sequences in $\Repn(A)$ to investigate $\wh{A}$.

The following lemma is known to experts but we provide a proof for completeness.

\begin{lemma}\label{lemma:kernel_subrepresentations}
    Let $\sigma$ be a representation of $A$ in $\mc{H}$. If $\sigma\simeq\bigoplus_{j\in I}\sigma_j$ where each $\sigma_j$ is a representation of $A$ in $\mc{H}_j$, then $\ker\sigma=\bigcap_{j\in I}\ker\sigma_j$.
\end{lemma}

\begin{proof}
    Let $U:\mc{H}\rightarrow\bigoplus_{j\in I}\mc{H}_j$ be the linear isomorphism which intertwines $\sigma$ and $\bigoplus_{j\in I}\sigma_j$, i.e., for all $a\in A$, $U\sigma(a)=\l[\bigoplus_{j\in I}\sigma_j(a)\r]U$.

    $(\sse)$ Let $a\in \ker\sigma$. Let $\xi_j\in \mc{H}_j$ and define $\xi=U^{-1}\l[\bigoplus_{j\in I}\xi_j\r]\in\mc{H}$. Then, by assumption, $\sigma(a)\xi=0$. Thus, since $U\sigma(a)=\l[\bigoplus_{j\in I}\sigma_j(a)\r]U$, we see that 
    \[0=U\sigma(a)\xi=\l(\bigoplus_{j\in I}\sigma_j(a)\r)U\,U^{-1}\l[\bigoplus_{j\in I}\xi_j\r]=  \l(\bigoplus_{j\in I}\sigma_j(a)\xi_j\r).\]
    Thus, we must have that $\sigma_j(a)\xi_j=0$ for all $j\in I$. Because each $\xi_j\in \mc{H}_j$ was arbitrary, $a\in \ker\sigma_j$ for all $j\in I$.

    $(\supseteq)$ Suppose $a\in \bigcap_{j\in I}\ker\sigma_j$. Let $\xi\in\mc{H}$. Define $P_j:\bigoplus_{j\in I}\mc{H}_j\rightarrow\mc{H}_j$ to be the canonical projection for each $j\in I$ and let $\xi_j=P_jU\xi$. Since $U$ is an isomorphism and the $P_j$ are mutually orthogonal projections with $\te{Id}=\sum_{j\in I}P_j$ (the identity on $\bigoplus_{j\in I}\mc{H}_j$), we see that \[U\xi=\te{Id}\,U\xi=\sum_{j\in I}P_jU\xi=\bigoplus_{j\in I}\xi_j.\]

    Thus,
    \[0=\l[\bigoplus_{j\in I}\sigma_j(a)\r]\bigoplus_{j\in I}\xi_j=\l[\bigoplus_{j\in I}\sigma_j(a)\r]U\xi=U\sigma(a)\xi.\]
    Since $U$ is a linear isomorphism, we must have $\sigma(a)\xi=0$. Because $\xi\in\mc{H}$ was arbitrary, we conclude that $a\in\ker\sigma$.

\end{proof}

%\begin{cor}
%    If $\rho\leq \pi$, then $\ker\pi\sse \ker\rho$.
%\end{cor}

%\begin{proof}
%    Write $\pi\simeq \rho\oplus\rho'$ and apply the lemma above.
%\end{proof}

\begin{cor}\label{cor:kernel_containment}
    If $\pi=\bigoplus_{j\in I}\pi_j$, then $\ker\pi=\bigcap_{j\in I}\ker\pi_j$.
\end{cor}

If $\pi$ be a representation of $A$ and $\mc{S},\mc{T}$ sets of representations of $A$, we say that $\pi$ is {\em weakly contained} in $\mc{S}$ if $\cap_{\rho\in \mc{S}}\ker \rho\sse \ker \pi$ and we say that $\mc{T}$ is {\em weakly contained} in $\mc{S}$ if each element of $\mc{T}$ is weakly contained in $\mc{S}$.

%\begin{thm}[\cite{Dix77} 3.4.10 (p.79)]\label{thm:equivalent_weak_containment}
%    Let $\pi\in \wh{A}$ and $\mc{S}\sse \wh{A}$. Then $\pi\in\ov{\mc{S}}$ if and only if $\pi$ is weakly contained in $\mc{S}$.
%\end{thm}

For the next result, we introduce some notation.
    Suppose $\{\pi_{\lambda}\}_{\lambda\in\Lambda}$ is a net of nondegenerate representations of $A$ on $\hn$. For each $\lambda$, we have $\pi_{\lambda}\simeq\bigoplus_{j_{\lambda}\in I_{\lambda}} \pi_{\lambda}^{j_{\lambda}}$ where each $\pi_{\lambda}^{j_{\lambda}}$ is irreducible.
    Define
    \[\mc{T}(\{\pi_{\lambda}\}):=\{\pi_{\lambda}^{j_{\lambda}}:\lambda\in\Lambda,j_{\lambda}\in I_{\lambda}\}\sse\wh{A}\]
    the collection of all the irreducible subrepresentations of the net $\{\pi_{\lambda}\}$. We will abuse notation and think of $\mc{T}(\pi)$ to be the collection of these irreducible representations as arising from the constant net $\{\pi\}$.

\begin{prop}\label{prop:T_weak_containment}
    Suppose a net $\pi_{\lambda}\rightarrow\pi\simeq\bigoplus_{j\in I}\pi^j$ in $\Repnp(A)$ for some cardinal $n$ where each $\pi^j$ is irreducible. Then $\mc{T}(\pi)$ is weakly contained in $\mc{T}(\{\pi_{\lambda}\})$.
\end{prop}

\begin{proof}
    We need to show that $\pi^j$ is weakly contained in $\mc{T}(\{\pi_{\lambda}\})$ for all $j\in I$, that is,
    \[\bigcap_{\lambda\in\Lambda}\bigcap_{j_{\lambda}\in I_{\lambda}}\ker\pi_{\lambda}^{j_{\lambda}}\sse\ker\pi^j\te{ for each }j\in I.\]
    Thus, we want
    \begin{equation}\label{weakcon}
    \bigcap_{\lambda\in\Lambda}\bigcap_{j_{\lambda}\in I_{\lambda}}\ker\pi_{\lambda}^{j_{\lambda}}\sse\bigcap_{j\in I}\ker\pi^j.
\end{equation}

    By Lemma \ref{lemma:kernel_subrepresentations}, we know that $\ker\pi_{\lambda}=\bigcap_{j_{\lambda}\in I_{\lambda}}\ker\pi_{\lambda}^{j_{\lambda}}$ and $\ker\pi=\bigcap_{j\in I}\ker\pi^j$. Rewriting equation (\ref{weakcon}), we must show
    \begin{equation}\label{weakcon2}
    \bigcap_{\lambda\in \Lambda}\ker\pi_{\lambda}\sse\ker\pi.
    \end{equation}

    Now, let $a\in \bigcap_{\lambda\in \Lambda}\ker\pi_{\lambda}$. Because we are assuming that $\pi_{\lambda}\rightarrow\pi$ in $\Repnp(A)$, for all $\xi\in \hn$,
    $$\|\pi_{\lambda}(a)\xi-\pi(a)\xi\|_{\hn}\rightarrow 0.$$
    Of course, $\pi_{\lambda}(a)\xi=0$ by choice and so this becomes
    $$\|\pi(a)\xi\|_{\hn}\rightarrow 0.$$
    Since $\pi(a)\xi$ is a constant sequence, we conclude that $\pi(a)\xi=0$. As $\xi\in\hn$ was arbitrary, we must have $a\in\ker\pi$ as desired.

\end{proof}

\begin{prop}\label{prop:sequence_to_subrepresentations}
    If $\pi_k\rightarrow\pi$ in $\Repn(A)$ where $\{\pi_k\}\sse\Irrn(A)$ and where $\pi\simeq\bigoplus_{j\in I}\pi^j$ for $\{\pi^j\}\sse {}_n\wh{A}$, then $\pi_k\rightarrow \pi^j$ in $\wh{A}$ for all $j\in I$.
\end{prop}

\begin{proof}
    Fix $j\in I$ and let $\mc{O}_j\sse\Prim(A)$ be an open set containing $\ker\pi^j$. Then, of course, $\mc{F}_j=\Prim(A)\setminus\mc{O}_j$ is closed in $\Prim(A)$ and does not contain $\ker\pi^j$. By Proposition 3.1.2 in \cite{Dix77}, there exists a subset $M\sse A$ such that $\mc{F}_j$ is the set of all primitive ideals of $A$ containing $M$. That is, let
    $$\mc{F}_j=\mc{I}(M):=\{\ker\rho\,:\,\rho\in\wh{A}\te{ and } M\sse \ker\rho\}.$$

    Because $\ker\pi^j\not\in\mc{F}_j=\mc{I}(M)$, we see that $M$ is not contained in $\ker\pi^j$. This means there exists $a\in M$ such that $a\not\in \ker\pi^j$. By Corollary \ref{cor:kernel_containment}, $\ker\pi=\bigcap_{j\in I}\ker\pi^j$ and so $a\not\in \ker\pi$. Thus, there exists $\xi_a\in \hn$ such that $c=\normspace{\hn}{\pi(a)\xi_a}>0$.

    By assumption, $\pi_k\rightarrow\pi$ in $\Repn(A)$. Fix $\ve=\dfrac{c}{2}>0$. Then there exists $K\in\zn_{>0}$ such that $k\geq K$ implies
    \begin{align*}
    &\normspace{\hn}{\pi_k(a)\xi_a-\pi(a)\xi_a}<\ve=\dfrac{c}{2}\\
    \Rightarrow\hspace{.15in}&\bigg|\normspace{\hn}{\pi_k(a)\xi_a}-\normspace{\hn}{\pi(a)\xi_a}\bigg|<\dfrac{c}{2}\\
    \iff\hspace{.15in}\dfrac{c}{2}&<\normspace{\hn}{\pi_k(a)\xi_a}
    \end{align*}
    Thus, we see that $\normspace{\hn}{\pi_k(a)\xi_a}>\dfrac{c}{2}>0$ for all $k\geq K$. In particular, this means that $a\not\in\ker\pi_k$ for all $k\geq K$ and consequently $M$ is not contained in any of these $\ker\pi_k$. Hence, $\ker\pi_k\not\in\mc{I}(M)=\mc{F}_j$ for all $k\geq K$. Of course, this means these $\ker\pi_k$ are contained in $\mc{O}_j=\Prim(A)\setminus\mc{F}_j$. 

    We conclude by the homeomorphism between $\wh{A}$ and $\Prim(A)$ that $[\pi_k]\rightarrow[\pi^j]$ in $\wh{A}$.

\end{proof}

Before we prove our next theorem, we include a useful result from Fell (though this particular phrasing is courtesy \cite{Dav96}).

\begin{lemma}[\cite{Fel60} Lemma 2.2 (p.380)]\label{lemma:net_convergence}
    A net $\sigma_{\alpha}$ in $\wh{A}$ has $\rho$ as a limit point if and only 
    $$\liminf_{\alpha}\|\sigma_{\alpha}(a)\|_{\te{op}}\geq \|\rho(a)\|_{\te{op}}\quad\te{ for all }a\in A.$$
\end{lemma}

\begin{thm}\label{theorem:sequence_to_subrepresentations}
    Let $A$ be a liminal $C^*$-algebra and fix $n\in\zn_{>0}$. Suppose $\pi_k\rightarrow\pi$ in $\Repn(A)$ such that $\{\pi_k\}\sse\Irrn(A)$ and where $\pi\simeq\bigoplus_{j\in I}\pi^j$ for $\{[\pi^j]\}\sse {}_n\wh{A}$. If $[\pi_k]\rightarrow [\pi']\in \wh{A}$, then there exists $j\in I$ such that $\pi'\simeq\pi^j$.
\end{thm}

\begin{proof}
    Firstly, that $\frsub{n}\wh{A}$ is closed in $\wh{A}$. Because $\{[\pi_k]\}\sse\wh{A}_n$ converges to $[\pi']$, we must have that $\dim\pi'\leq n$ by closedness.

    Secondly, suppose that $\pi'\not\simeq \pi^j$ for any $j\in I$. This means that $\ker\pi'\neq\ker\pi^j$ for all $j\in I$. We observe that, by the proof of Proposition \ref{prop:sequence_to_subrepresentations}, 
    $$\bigcap_k\ker\pi_k\sse\ker\pi^j\cap \ker \pi'\te{ for all }j\in I.$$

    If $\ker\pi^j\sse\ker \pi'$ for some $j\in I$, then Corollary \ref{cor:closed_points_app} implies $\pi^j\simeq \pi'$. So instead, we assume for each $j\in I$ that there exists $a_j\in \ker\pi^j$ such that $a_j\not\in \ker \pi'$. Define $a=\prod_{j\in I} a_j$ (where we note that $I$ is a necessarily finite set). Then $a\in \ker\pi^j$ for all $j\in\ I$ and $a\not\in \ker\pi'$ because primitive ideals are prime (\cite[II.6.5.15 (p.115)]{Bla10}). In particular, we note that $\pi^j(a)\xi=0\te{ for all }\xi\in \mc{H}_n$ and there exists $\xi_0\in\mc{H}_n$ such that $\pi'(a)\xi_0\neq 0$.

    Since $\pi_k\rightarrow \pi\in \Repn(A)$, for all $\xi\in \mc{H}_n$, 
    \begin{align*}
    \|\pi_k(a)\xi-\pi(a)\xi\|_{\hn}&=\|\pi_k(a)\xi\|_{\hn}\rightarrow 0.
    \end{align*}
    Hence, $\liminf_k\|\pi_k(a)\|_{\te{op}}=0$. However, $\|\pi'(a)\xi_0\|_{\hn}>0$ and so $\|\pi'(a)\|_{\te{op}}>0$. Lemma \ref{lemma:net_convergence} implies that $[\pi_k]$ does not converge to $[\pi']$ in $\wh{A}$, a contradiction.

    Therefore, we must have $\ker\pi'=\ker\pi^j$ for some $j\in I$, i.e., $\pi'\simeq \pi^j$.

\end{proof}

\subsection{How Convergence in \str{$\wh{N}$}{N-hat} Influences Convergence in \str{$\wh{G}$}{G-hat}}\label{subsec:convergence_N_hat}

We want to specialize the investigation to the topology of unitary duals arising from crystallographic groups (though the arguments should also follow for discrete, finitely generated, virtually abelian groups). We begin by establishing that the associated group $C^*$-algebras are separable and liminal. Note that Thoma's famed result about virtually abelian groups (see \cite{Tho68}) ensures that $C^*(G)$ is Type I (which also follows whenever $C^*(G)$ is liminal).

Because each crystallographic group $G$ is finitely presented, they are all countable in addition to being discrete and are, thus, separable. This means that all associated $C^*(G)$ are separable (\cite[13.9.2 (p.303)]{Dix77}). 

We say a $C^*$-algebra $A$ is {\em subhomogeneous} if there exists $M\in\zn_{>0}$ such that $\wh{A}={}_M\wh{A}$. In particular, we note that subhomogeneous $C^*$-algebras are liminal.

\begin{thm}[\cite{Moo72}, Thm 1 (p.402)]\label{thm:bounded_representations}
    If $G$ is a locally compact group, then there is an integer $M$ such that $\dim \pi\leq M<\infty$ for all $\pi\in\wh{G}$ if and only if there is an open abelian subgroup of finite index in $G$.
\end{thm}

Putting this together yields the following corollary.

\begin{cor}\label{cor:crystallography_group_liminal}
    If $G$ is a crystallographic group, then $C^*(G)$ is separable and subhomogeneous. In particular, it is liminal and Type I.
\end{cor}

\begin{rmk}
We believe many of the following arguments can be generalized to finitely presented discrete virtually abelian groups. This extension is the focus of a future project.
\end{rmk}

We also take a moment to note that, because our $C^*$-algebras are subhomogeneous, we are guaranteed to not miss vital spectral information by only focusing on finite dimensional representations. 

Now, we want to investigate how the topology of $\wh{N}$ impacts the topology of $\wh{G}$. First, we rephrase Proposition \ref{prop:sequence_to_subrepresentations} in terms of $\Repn(G)$ and $\wh{G}$. Afterwards, we focus on sequences in $\wh{N}$.

\begin{prop}
    Let $\{\pi_k\}\sse\Irrn(G)$ and $\pi\in\Repn(G)$ where $\pi\simeq \bigoplus_{j\in I}\pi^j$ ($\{[\pi^j]\}\sse{}_n\wh{G}$). If, for all $\xi\in\hn$ and $s\in G$, we have
    $$\normspace{\hn}{\pi_k(s)\xi-\pi(s)\xi}\rightarrow 0,$$
    then $[\pi_k]\rightarrow [\pi^j]$ in $\wh{G}$ for all $j\in I$.
\end{prop}

\begin{proof}
    Combine Propositions \ref{prop:facts_list} and \ref{prop:sequence_to_subrepresentations} for $A=\cstarg$.
\end{proof}

\begin{prop}\label{prop:conv_chars}
    Suppose $\chi_k\rightarrow\chi\in\wh{N}\cong\tn^{r}$, then $G_{\chi_k}\leq G_{\chi}$ eventually.
\end{prop}

\begin{proof}\mbox{}
    Because $G$ contains at most a finite number of subgroups containing $N$, we have a finite number of choices for stabilizers of characters of $N$. Thus, we may pass to a subsequence such that $G_{\chi_k}$ is fixed for all $k$. Fix $g\in G_{\chi_k}$ and $a\in N$. Write
        \begin{align*}
        |g\cdot\chi(a)-\chi(a)|&=|g\cdot\chi(a)-g\cdot\chi_k(a)+g\cdot\chi_k(a)-\chi(a)|\\
        &\leq|g\cdot\chi(a)-g\cdot\chi_k(a)|+|g\cdot\chi_k(a)-\chi(a)|\\
        &=|\chi(g^{-1}\cdot a)-\chi_k(g^{-1}\cdot a)|+|\chi_k(a)-\chi(a)|\longrightarrow 0
        \end{align*}
    Thus, $g\cdot\chi(a)=\chi(a)$. 
    We observe that $N$ is finitely generated, so we may conclude that $g$ fixes $\chi$ on all generators of $N$  {\em simultaneously} and, therefore, on all of $N$. Hence, $g\in G_{\chi}$.

\end{proof}

Let $\chi\in\wh{N}$ and $|\chi|:=|\mc{O}_{\chi}|$. A quick application of Proposition \ref{prop:conv_chars} is the following.

\begin{cor}
If $\chi_k\rightarrow\chi\in\wh{N}$, then $|\chi|$ divides $|\chi_k|$ eventually.
\end{cor}

Suppose we have $\pi\in\Irrn(G)$. Let $P_1:\hn\rightarrow\hn$ be the projection given by $P_1\xi=(\xi_1,0,...,0)$ where $\xi=(\xi_1,\xi_2,...,\xi_n)\in\hn$ and let $e_1=(1,0,...,0)\in\hn$. Define the {\em $\chi$ associated to the concrete representation $\pi$} by the formula
\[\chi(h):=P_1\pi(h)e_1\te{ for all }h\in N.\]

We note that $\chi$ is indeed a character by Remark \ref{rmk:restriction_induced_reps}.

\begin{prop}\label{prop:recovering_chi}
Suppose $\pi_k\rightarrow \pi$ in $\Repn(G)$ where $\{\pi_k\}\sse\Irrn(G)$. Let $\chi_k$, $\chi$ be the characters associated to the concrete representations $\pi_k$, $\pi$ respectively. Then $\chi_k\rightarrow\chi$ in $\wh{N}$.
\end{prop}

\begin{proof}
    Let $\{\pi_k\}\sse\Irrn(G)$ converge to some possibly reducible $\pi\in\Repn(G)$. Lemma \ref{lemma:limit_representations_is_representation} applied to $\chi_k\rightarrow \chi$ shows that $\chi\in\wh{N}$.
    
    Let $h\in N$. Proposition \ref{prop:facts_list} implies that
    \begin{align*}
    |\chi_k(h)-\chi(h)|&=\|P_1\pi_k(h)e_1-P_1\pi(h)e_1\|_{\hn}\\
    &=\|P_1(\pi_k(h)e_1-\pi(h)e_1)\|_{\hn}\\
    &\leq \|P_1\|_{\hn}\|\pi_k(h)e_1-\pi(h)e_1\|_{\hn}\\
    &\leq \|\pi_k(h)e_1-\pi(h)e_1\|_{\hn} \longrightarrow 0.
    \end{align*}

\end{proof}

\begin{defn}
Let $B(\chi,\ve)$ be an $\ve$-neighborhood of $\chi$ in $\wh{N}\cong\tn^r$. Let
\[\wh{N}_k:=\{\chi\in\wh{N}:|\chi|=k\}\]
denote the set of {\em $k$-orbits}. We say that $\chi\in\wh{N}$ is a {\em $k$-boundary point} if, for all $\ve>0$, there exists $\chi'\in\wh{N}_k$ such that $\chi'\in B(\chi,\ve)$. $\chi$ is a {\em $k$-interior point} if there exists $\ve>0$ such that $B(\chi,\ve)\sse\wh{N}_k$. We say $\chi$ is a {\em $k$-orbit drop point} if $\chi$ is a $k$-boundary point and $|\chi|<k$.
\end{defn}

We observe that Proposition \ref{prop:conv_chars} implies the set ${}_k\wh{N}=\bigcup_{j=1}^k\wh{N}_j$ is closed in $\wh{N}$; equivalently, $\wh{N}\setminus {}_k\wh{N}$ is open in $\wh{N}$.

\begin{cor}\label{cor:orbit_drop_points}
If $\{\pi_k\}\sse\Irrn(G)$ and $\pi_k\rightarrow\pi$ in $\Repn(G)$ where $\pi$ is reducible, then $\chi$ is an orbit drop point.
\end{cor}

\begin{proof}
    Suppose $\chi$ is not an orbit drop point. Proposition \ref{prop:recovering_chi} implies $\chi_k\rightarrow \chi$ and so $G_{\chi_k}\leq G_{\chi}$ eventually by Proposition \ref{prop:conv_chars}. Since $\chi$ is not an orbit drop point, we must have $G_{\chi_k}=G_{\chi}$ eventually. The sequence described in equation (\ref{eq:sequence2}) combined with Proposition \ref{prop:induction_facts} (\ref{prop:induction_direct_sums}) yields that $\pi$ must be irreducible.
\end{proof}

\begin{rmk}\label{rmk:orbit_drop_meaning}
Let $O_k$ be the set of $k$-orbit drop points (which we observe may be empty). Then it is exactly over set $\bigcup_{k=1}^{|D|}O_k$ where we would expect the unitary dual to exhibit non-Hausdorff behavior per Remark \ref{rmk:hausdorff_justification}. In particular, we observe that any sequence of irreducible representations built over $k$-interior points must {\em always} converge to an irreducible representation.
\end{rmk}

\begin{defn}\label{defn:k_path_connected}
We say $\chi\in\wh{N}_k$ is {\em $k$-path connected} to $\chi'\in\wh{N}_k$ if there exists a continuous function $h:[0,1]\rightarrow\wh{N}_k$ such that $h(0)=\chi$ and $h(1)=\chi'$ and we write $\chi\approx \chi'$.
\end{defn}

\begin{prop}\label{prop:path_stabilizer_subgroup}
If $h:[0,1]\rightarrow \wh{N}_k$ is a continuous function, then $G_{h(t)}=G_{h(s)}$ for all $s,t\in [0,1]$. In particular, if $\chi\approx \chi'$, then $G_{\chi}=G_{\chi'}$.
\end{prop}

\begin{proof}
    Let $h:[0,1]\rightarrow\wh{N}_k$ be a continuous function. Because there are only a finite number of $k$-index subgroups of $G$, we can partition the path $H=h([0,1])\sse\wh{N}_k$ into a finite number of equivalence classes $[[\rho]]$ where $\rho'\in[[\rho]]$ if and only if $G_{\rho}=G_{\rho'}$. Enumerate these equivalence classes $\{[[\rho_1]],[[\rho_2]],...,[[\rho_{m}]]\}$.
    
    Observe that Proposition \ref{prop:conv_chars} and the continuity of $h$ implies that if $\sigma\in H$ and if, for all $\ve>0$, there exists $\sigma_{\ve}\in B(\sigma,\ve)\cap H$ such that $\sigma_{\ve}\in[[\rho_j]]$ for some $1\leq j\leq m$, then $\sigma\in [[\rho_j]]$ as well. Because there are only a finite number of equivalence classes, we can always arrange for any infinite sequence of $\{\sigma_{\ve}\}$ to be arising from exactly one equivalence class by passing to a subsequence.
    
    Let $\sigma\in H$. Then, of course, $\sigma\in[[\rho_j]]$ for some $1\leq j\leq m$. The previous comment shows the set $\{\ve>0:B(\sigma,\ve)\cap H\sse[[\sigma]]=[[\rho_j]]\}$ is non-empty. We observe that $H$ is compact so there exists a finite collection $\{B(\sigma_{\ell},\ve_{\sigma_{\ell}})\}_{\ell=1}^{L}$ such that $H\sse \bigcup_{\ell=1}^LB(\sigma_{\ell},\ve_{\sigma_{\ell}})$. Because $H$ is connected, every $\rho\in H$ must be covered by at least two neighborhoods $B(\sigma_{\ell},\ve_{\sigma_{\ell}})$ and $B(\sigma_{j},\ve_{\sigma_{j}})$ for $\ell\neq j$. The discussion above would then imply that $[[\sigma_{\ell}]]=[[\sigma_j]]$. A bootstrapping argument then gives that $\sigma_{\ell}\in[[\sigma_1]]$ for all $1\leq \ell \leq L$.

    If $\chi\approx \chi'$, then there exists a continuous function $h:[0,1]\rightarrow\wh{N}_k$ such that $h(0)=\chi$ and $h(1)=\chi'$. Hence, $G_{\chi}=G_{\chi'}$ by the above discussion.
\end{proof}

\begin{rmk}
The converse of this proposition is false in general. See, for example, the 1-orbits described in Section \ref{example:one_orbits}.   
\end{rmk}

\begin{cor}\label{cor:k_path_connected_implies_cohomologous}
If $\chi\approx \chi'$, then $\chi\sim \chi'$ over their shared stabilizer modulo $N$.
\end{cor}

\begin{proof}
    In the proof of Proposition \ref{prop:path_stabilizer_subgroup}, replace Proposition \ref{prop:conv_chars} by Corollary \ref{cor:cohomologous_sequence} and replace the finite number of $k$-index subgroups of $G$ with the finite number of elements in $\msc{H}^2(D_{\chi},\tn)$.

\end{proof}

\subsection{Main Topological Results}\label{subsec:main_top_results}
We begin with a useful lemma already known to experts. The proof is based on comments made on p.8 of \cite{CST22}.

\begin{lemma}\label{lemma:inducing_preserves_convergence}
Suppose $\sigma_k\rightarrow \sigma\in \Repn(H)$. Let $H\leq K$ and $\mc{H}_{H}^K$ be the induced Hilbert space. Then $\textnormal{ind}_{H}^K\sigma_k\rightarrow \textnormal{ind}_{H}^K\sigma$ in $\usf(\mc{H}_{H}^K)$.
\end{lemma}

\begin{proof}
     
    We view each $\sigma_k,\sigma$ as matrices on the shared Hilbert space $\mc{H}_H$ of dimension $n$ such that $\sigma_k\rightarrow\sigma$ in $\usf(\mc{H}_H)$. Let $\mc{H}_H^K$ be the induced Hilbert space and consider $\pi_k=\textnormal{ind}_H^K\sigma_k$ and $\pi=\textnormal{ind}_H^K\sigma$. Note that $\mc{H}_H^K$ is isometrically isomorphic to $\bigoplus_{x\in K\slash H}\mc{H}_H$. Let $\mc{T}=\{t_1,t_2,..,t_d\}$ be a transversal of $K\slash H$. Then, when we represent $\pi_k$, $\pi$ on $\bigoplus_{x\in K\slash H}\mc{H}_H$, we produce the following block matrices: for $g\in K$, 
    \[\pi_k(g)\simeq\l[\begin{array}{cccc}
    \sigma_k(t_1^{-1}gt_1) & \sigma_k(t_1^{-1}gt_2) & \cdots & \sigma_k(t_1^{-1}gt_d) \\
    \sigma_k(t_2^{-1}gt_1)  & \sigma_k(t_2^{-1}gt_2)  & \cdots & \sigma_k(t_2^{-1}gt_d) \\
    \vdots & \vdots & \ddots & \vdots\\
    \sigma_k(t_d^{-1}gt_1) & \sigma_k(t_d^{-1}gt_2) & \cdots & \sigma_k(t_d^{-1}gt_d) 
    \end{array}\r]\]
    and
    \[\pi(g)\simeq\l[\begin{array}{cccc}
    \sigma(t_1^{-1}gt_1) & \sigma(t_1^{-1}gt_2) & \cdots & \sigma(t_1^{-1}gt_d) \\
    \sigma(t_2^{-1}gt_1)  & \sigma(t_2^{-1}gt_2)  & \cdots & \sigma(t_2^{-1}gt_d) \\
    \vdots & \vdots & \ddots & \vdots\\
    \sigma(t_d^{-1}gt_1) & \sigma(t_d^{-1}gt_2) & \cdots & \sigma(t_d^{-1}gt_d) 
    \end{array}\r]\]
    where we set $\sigma_k(t_j^{-1}gt_{\ell})$, $\sigma(t_j^{-1}gt_{\ell}):=0$ when $t_j^{-1}gt_{\ell}\not\in H$.
    
    Since $\sigma_k\rightarrow\sigma$ in $\usf(\mc{H}_H)$, the isometric isomorphism between $\mc{H}_H^K$ and $\bigoplus_{x\in K\slash H}\mc{H}_H$ implies we must have $\pi_k\rightarrow \pi$ in $\usf(\mc{H}_H^K)$.

\end{proof}

\begin{rmk}
The following main theorems pay homage to results of Baggett and Raeburn (\cite{Bag68},\cite{Rae82}) but we formulate them specifically for crystallographic groups and concrete sequences of irreducibles.
\end{rmk}

Let $\chi_k\rightarrow\chi\in\wh{N}$ where we assume the $\{\chi_k\}$ have a shared stabilizer, denoted $G_*$. By the discussion surrounding equation (\ref{eq:sequence2}) in Section \ref{subsec:constructing_convergent_sequences}, we can build a convergent sequence of irreducible representations of $\wh{G}_*^{(\chi_k)}$ in $\hn$, say $\sigma_k$, which converges to some irreducible representation in $\wh{G}_*^{(\chi)}$, say $\sigma$. This, of course, does not mean that $\sigma\in\wh{G}_{\chi}^{(\chi)}$. Note that this discussion also shows that if we have a sequence of irreducible representations with this behavior, it must arise from this construction. We call such a sequence the {\em generic sequence of irreducible representations over $\chi_k\rightarrow\chi$}.

This leads us to one of our main results.

\begin{thm}\label{thm:main_result}
    Let $G$ be a crystallographic group with lattice $N$ and point group $D$. 

    Let $\sigma_k\rightarrow \sigma$ be a generic sequence of irreducible representations over $\chi_k\rightarrow\chi$. Then the sequence $\{[\textnormal{ind}_{G_*}^G\,\sigma_k]\}\sse \wh{G}$ converges to every element of the set $\{[\textnormal{ind}_{G_{\chi}}^G\,\rho_j]\}\sse\wh{G}$ for some (nonempty) subset $\{[\rho_j]\}\sse \wh{G}_{\chi}^{(\chi)}$. Moreover, if $[\textnormal{ind}_{G_*}^G\,\sigma_k]\rightarrow[\pi]\in \wh{G}$, then $\pi\simeq \on{ind}_{G_{\chi}}^G\rho$ for some $[\rho]\in\wh{G}_{\chi}^{(\chi)}$.
\end{thm}

\begin{proof}
    Represent each $\sigma_k,\sigma$ on a shared Hilbert space $\mc{H}_{G_*}$. On the induced Hilbert space, $\mc{H}^{G_{\chi}}_{G_*}$, concretely consider each $\textnormal{ind}_{G_*}^{G_{\chi}}\,\sigma_k$. Lemma \ref{lemma:inducing_preserves_convergence} implies this sequence will converge to $\textnormal{ind}_{G_*}^{G_{\chi}}\,\sigma$ in $\usf(\mc{H}^{G_{\chi}}_{G_*})$. Remark \ref{rmk:restriction_induced_reps} indicates that $(\textnormal{ind}_{G_*}^{G_{\chi}}\sigma)|_N=\chi^{\oplus m}$ where $m=[G_{\chi}:G_*]\cdot \dim\sigma$ and so $\textnormal{ind}_{G_*}^{G_{\chi}}\,\sigma$ is a (possibly reducible) representation of $G_{\chi}$ which restricts to $\chi$. Thus, it is the direct sum of elements of $\wh{G}_{\chi}^{(\chi)}$. Let $\textnormal{ind}_{G_*}^{G_{\chi}}\,\sigma\simeq \bigoplus_{j=1}^{\ell}\rho_j^{\oplus m_j}$ where $[\rho_j]\in\wh{G}_{\chi}^{(\chi)}$ (chosen to be pairwise inequivalent) and $m_j\in\zn_{>0}$. Then, by Proposition \ref{prop:induction_facts} (\ref{prop:induction_direct_sums}, \ref{prop:induction_in_stages}),
     \[      \textnormal{ind}_{G_{\chi}}^G\,\textnormal{ind}_{G_*}^{G_{\chi}}\,\sigma\simeq\textnormal{ind}_{G_{\chi}}^G\l(\bigoplus_{j=1}^{\ell}\rho_j^{\oplus m_j}\r)    \simeq\bigoplus_{j=1}^{\ell}\l(\textnormal{ind}_{G_{\chi}}^G\,\rho_j\r)^{\oplus m_j}.
    \]
    Because $[\rho_j]\in \wh{G}_{\chi}^{(\chi)}$, $[\textnormal{ind}_{G_{\chi}}^G\,\rho_j]\in \wh{G}$ by the Mackey Machine (Theorem \ref{thm:Mackey_machine}). 

    By Proposition \ref{prop:induction_facts} (\ref{prop:induction_in_stages}), $\textnormal{ind}_{G_{\chi}}^G\textnormal{ind}_{G_*}^{G_{\chi}}\,\sigma\simeq\textnormal{ind}_{G_{*}}^G\,\sigma.$ Lemma \ref{lemma:inducing_preserves_convergence} implies $\textnormal{ind}_{G_*}^G\,\sigma_k\rightarrow \textnormal{ind}_{G_*}^G\,\sigma$, where the latter is unitarily equivalent to $\bigoplus_{j=1}^{\ell}\l(\textnormal{ind}_{G_{\chi}}^G\,\rho_j\r)^{\oplus m_j}$. The result follows by Proposition \ref{prop:sequence_to_subrepresentations} and Theorem \ref{theorem:sequence_to_subrepresentations}.

\end{proof}

Theorem \ref{thm:main_result} indicates {\em where} we should start looking for the irreducible subrepresentations of the limit of a sequence of irreducible representations. However, not every limit is simply the direct sum of {\em all} the irreducible representations living over the associated character (see, for instance, Example \ref{ex:4_orbit_to_1_orbit}). Thus, there remains the question of how to determine which irreducibles are actually subrepresentations of the limit.

To answer this question, we appeal to the character theory of projective representations. Amazingly, the character theory of projective representations directly mimics character theory of finite groups \cite[Cor~7.15 (p.199)]{CST22}. If $\sigma^P,\rho^P$ are $\eta$-representations of a finite group $H$, we define
\[\langle \sigma^P,\rho^P\rangle_H=\dfrac{1}{|H|}\sum_{h\in H}\Tr(\sigma^P(h))\,\ov{\Tr(\rho^P(h))}.\]
In particular, if $\sigma^P\simeq\bigoplus_{j=1}^k(\rho_j^P)^{\oplus m_j}$, then \label{prop:character_theory}
\[\langle \sigma^P,\rho^P_j\rangle_H = m_j.\]
We caution that not only must $\sigma^P,\rho^P$ be projective representations of the same finite group, but they are both arising from the same 2-cocycle $\eta$.

Our goal is to extend this result to ordinary representations of crystallographic groups. 

Recall from Theorem \ref{thm:projective_to_regular_bijection} that there is a dimension preserving bijection from $D_{\chi}^{\omega_{\chi}}$ to $\wh{G}_{\chi}^{(\chi)}$ given by $\pi^P\mapsto\sigma$ where
\[\sigma(x)=\chi^*(x)\cdot\pi^P(q(x)).\]
Hence, we see that for all $x\in G_{\chi}$, 
\[\Tr(\sigma(x))=\Tr(\chi^*(x)\cdot \pi^P(x'))=\chi^*(x)\cdot\Tr(\pi^P(x'))\te{ where }x'=q(x).\]
We note that $\chi^*:G_{\chi}\rightarrow\tn$ and so $\chi^*\cdot\ov{\chi^*}=1$. We also note that Proposition \ref{prop:projective_lift} ensures that the isotypic components of $\sigma$ are preserved under this bijection. 
 Thus, if $[\rho_j]\in\wh{G}_{\chi}^{(\chi)}$ and $\sigma\simeq\bigoplus_{j=1}^k\rho_j^{\oplus m_j}$, then we may define
\[\langle \sigma,\rho_j\rangle_{G_{\chi}}:=\dfrac{1}{|D_{\chi}|}\sum_{h\in \mc{T}_{\chi}}\Tr(\sigma(h))\,\ov{\Tr(\rho_j(h))}=m_j\]
where $\mc{T}_{\chi}$ is a transversal of $G_{\chi}\slash N$. 

This discussion combined with the proof of Theorem \ref{thm:main_result} gives the following result. 

\begin{thm}\label{thm:character_theory}
Let $G$ be a crystallographic group with lattice $N$ and point group $D$. 

Suppose $\sigma_k\rightarrow\sigma$ is a generic sequence of irreducible representations over $\chi_k\rightarrow\chi$. Let $\{[\rho_j]\}_{j=1}^{\ell}$ be inequivalent irreducible concrete representations of the elements of $\wh{G}_{\chi}^{(\chi)}$ (where $|\wh{G}_{\chi}^{(\chi)}|=\ell$), each with associated character $\chi$. Then, if $\langle \sigma,\rho_j\rangle_{G_{\chi}}=m_j$, 
\[\textnormal{ind}_{G_{\chi}}^{G}\sigma\simeq\bigoplus_{j=1}^{\ell}(\textnormal{ind}_{G_{\chi}}^G\rho_j)^{\oplus m_j}.\]

\end{thm}

\subsection{Investigating the Topology via Sequences}\label{subsec:unitary_dual_via_sequences}

Based on Sections \ref{subsec:the_program}, \ref{subsec:constructing_convergent_sequences} and Theorems \ref{thm:main_result}, \ref{thm:character_theory}, we propose the following strategy for investigating the topology of the unitary dual of a crystallographic group.

\tbf{The Strategy:}
\begin{enumerate}
    \item Choose an $k$-orbit drop $\chi\in\wh{N}$ and $\{\chi_n\}\sse\wh{N}_k$ a sequence with shared stabilizer $G_{*}$ such that $\{\chi_n\}\rightarrow\chi\in\wh{N}$ (see Remark \ref{rmk:orbit_drop_meaning} for rationale). Proposition \ref{prop:conv_chars} implies $G_{*}\leq G_{\chi}$ and Corollary \ref{cor:orbit_drop_points} gives $G_{*}\neq G_{\chi}$.
   
    \item Produce cocycles $\omega_n\colon D_*\times D_*\to\T$ and $\omega\colon D_{\chi}\times D_{\chi}\to\T$, given by $\omega_n:=\ol{\chi_n}\circ\nu$ and $\omega:=\ol{\chi}\circ\nu$. Observe that $\omega_n\to\omega$ over $D_*$. Corollaries \ref{cor:cohomologous_sequence} and \ref{cor:k_path_connected_implies_cohomologous} gives that $\{\omega_n\}$, $\omega$ share a cohomology class over $D_*$.
  
    \item For each $n\geq1$, define $\lambda_n\colon D_*\to\T$ such that $\lambda_n(1)=1$ and $\omega_n=\omega_0\cdot\tau_{\lambda_n}$.
    
    \item Reduce to a subsequence $\chi_n\to\chi$, where $\lambda_n\to\lambda$ (which we can do because $\tn$ is compact and $D_*$ is finite). Over $D_*$, we have $\omega=\omega_0\cdot\tau_{\lambda}$ since $\omega_n\to\omega$.
    
    \item Fix a Hilbert space $\mc{H}$ such that we have a sequence of representations $\sigma_n^P\in\wh{D_*}^{\omega_n}$ on $\mc{H}$ converging to some $\omega$-representation $\sigma^P$ on $\mc{H}$.
    
    \item Reduce to a subsequence $\chi_n\to\chi$ where there exists $\Phi^P\in\wh{D_*}^{\omega_0}$ such that $\sigma_0^P=\Phi^P$ and $\sigma_n^P=\Phi^P\cdot\lambda_n$, for $n>0$. It must be that $\sigma^P=\Phi^P\cdot\lambda$, since $\sigma_n^P\to\sigma^P$ on $\mc{H}$. This implies that $\sigma^P\in\wh{D_*}^{\omega}$, emphasis on irreducible.
    
    \item Lift the sequence of projective representations up to domain $G_*$ so that we have $\chi_0^*\cdot\Phi^P,\ \chi_1^*\cdot\Phi^P\cdot\lambda_1,\ \chi_2^*\cdot\Phi^P\cdot\lambda_2,\dots\to\chi^*\cdot\Phi^P\cdot\lambda$, suppressing $\Phi^P$ and $\lambda_n$ to mean $\Phi^P\circ q$ and $\lambda_n\circ q$, respectively. They are all ordinary representations over $G_*$, and each of the $\chi_n^*\cdot\Phi^P\cdot\lambda_n$ and $\chi^*\cdot\Phi^P\cdot\lambda$ are irreducible; hence, they are elements of $\wh{G_*}^{(\chi_n)}$ and $\wh{G_*}^{(\chi)}$, respectively.
    
    \item Induce from $G_*$ to $G_{\chi}$ to produce the convergent sequence: $$\on{ind}_{G_*}^{G_{\chi}}(\chi_0^*\cdot\Phi^P),\ \on{ind}_{G_*}^{G_{\chi}}(\chi_1^*\cdot\Phi^P\cdot\lambda_1),\ \on{ind}_{G_*}^{G_{\chi}}(\chi_2^*\cdot\Phi^P\cdot\lambda_2),\dots\to\on{ind}_{G_*}^{G_{\chi}}(\chi^*\cdot\Phi^P\cdot\lambda)\te{ in }\usf(\mc{H}_{G_*}^{G_{\chi}}).$$ The ordinary representation $\pi:=\on{ind}_{G_*}^{G_{\chi}}(\chi^*\cdot\Phi^P\cdot\lambda)$ is reducible because $G_*\lneq G_{\chi}$. 
    
    \item Use Theorem \ref{thm:character_theory} to write
    \[\pi=\on{ind}_{G_*}^{G_{\chi}}(\chi^*\cdot\Phi^P\cdot\lambda)\simeq\bigoplus_{j=1}^{\ell}\rho_j^{\oplus m_j}\]
    for $[\rho_j]\in\wh{G}_{\chi}^{(\chi)}$ where $m_j=\langle \pi,\rho_j\rangle_{\chi}$.
    
    \item Let $\pi_n:=\on{ind}_{G_{\chi}}^G\on{ind}_{G_*}^{G_{\chi}}(\chi_n^*\cdot\Phi^P\cdot \lambda_n)$ where we note all convergent sequences of irreducible representations arise in this way according to the discussion surrounding equation (\ref{eq:sequence2}) in Section \ref{subsec:constructing_convergent_sequences}. Apply Theorem \ref{thm:main_result} to conclude $\pi_n\rightarrow \on{ind}_{G_{\chi}}^G\rho_j^{\oplus m_j}$ whenever $m_j\neq 0$ and everything to which $\{\pi_n\}$ converges is of this form.
\end{enumerate}

%%%%%%%%%%%%%%%%%%%%%%%%%%%%%%%%%%%%%%%%%%%%
%%%%%%%%%%%%%%%%%%%%%%%%%%%%%%%%%%%%%%%%%%%%

\section{Investigating the Spectrum of Dimension 3 Crystallographic Group 90}\label{sec:group_90}

\subsection{Dimension 3 Crystallographic Group 90}\label{subsec:group_90}
In this section, we include proof of concept examples demonstrating the application of the code and the conclusions of Theorems \ref{thm:main_result} and \ref{thm:character_theory}. We investigate the unitary dual of the dimension 3 crystallographic group 90. Note that the numbering for dimension 3 crystallographic groups (frequently referred to as ``space groups") is standardized and is used in fields such as chemistry, physics, and crystallography. We begin our examination with following data as provided by GAP's (\cite{GAP4}) group library and the GAP package CrystCat (\cite{CrystCat}):
\[G=\langle a,b,c\,|\, (b^{-1}a)^2=a^4=b^{-1}cbc=c^{-1}a^{-1}ca=(a^{-2}b^{-2})^2\rangle\]
\[\zn^3\cong N=\langle ab^2a^{-1},b^{-2},c\rangle\]
\[D=\{e,a,a^2,a^3,b,ab,a^2b,a^3b\}\]
\[1\rightarrow\zn^3\rightarrow G\rightarrow D\rightarrow 1\]
We list and label the lattice of subgroups of $D$.
% https://q.uiver.app/#q=WzAsMTAsWzIsMywiXFx7ZVxcfSJdLFsyLDAsIkQ9XFx7ZSxhLGFeMixhXjMsYixhYixhXjJiLGFeM2JcXH0iXSxbMiwxLCJEXzI9XFx7ZSxhLGFeMixhXjNcXH0iXSxbMSwxLCJEXzE9XFx7ZSxhXjIsYixhXjJiXFx9Il0sWzMsMSwiRF8zPVxce2UsYV4yLGFiLGFeM2JcXH0iXSxbMiwyLCJEXzY9XFx7ZSxhXjJcXH0iXSxbMCwyLCJEXzQ9XFx7ZSxhXjJiXFx9Il0sWzEsMiwiRF81PVxce2UsYlxcfSJdLFszLDIsIkRfNz1cXHtlLGFiXFx9Il0sWzQsMiwiRF84PVxce2UsYV4zYlxcfSJdLFswLDUsIiIsMCx7InN0eWxlIjp7ImhlYWQiOnsibmFtZSI6Im5vbmUifX19XSxbMCw2LCIiLDIseyJzdHlsZSI6eyJoZWFkIjp7Im5hbWUiOiJub25lIn19fV0sWzAsNywiIiwyLHsic3R5bGUiOnsiaGVhZCI6eyJuYW1lIjoibm9uZSJ9fX1dLFswLDgsIiIsMix7InN0eWxlIjp7ImhlYWQiOnsibmFtZSI6Im5vbmUifX19XSxbMCw5LCIiLDIseyJzdHlsZSI6eyJoZWFkIjp7Im5hbWUiOiJub25lIn19fV0sWzEsMiwiIiwyLHsic3R5bGUiOnsiaGVhZCI6eyJuYW1lIjoibm9uZSJ9fX1dLFsxLDMsIiIsMCx7InN0eWxlIjp7ImhlYWQiOnsibmFtZSI6Im5vbmUifX19XSxbMSw0LCIiLDAseyJzdHlsZSI6eyJoZWFkIjp7Im5hbWUiOiJub25lIn19fV0sWzUsMiwiIiwwLHsic3R5bGUiOnsiaGVhZCI6eyJuYW1lIjoibm9uZSJ9fX1dLFs1LDMsIiIsMCx7InN0eWxlIjp7ImhlYWQiOnsibmFtZSI6Im5vbmUifX19XSxbNywzLCIiLDIseyJzdHlsZSI6eyJoZWFkIjp7Im5hbWUiOiJub25lIn19fV0sWzYsMywiIiwyLHsic3R5bGUiOnsiaGVhZCI6eyJuYW1lIjoibm9uZSJ9fX1dLFs4LDQsIiIsMix7InN0eWxlIjp7ImhlYWQiOnsibmFtZSI6Im5vbmUifX19XSxbOSw0LCIiLDIseyJzdHlsZSI6eyJoZWFkIjp7Im5hbWUiOiJub25lIn19fV0sWzUsNCwiIiwxLHsic3R5bGUiOnsiaGVhZCI6eyJuYW1lIjoibm9uZSJ9fX1dXQ==
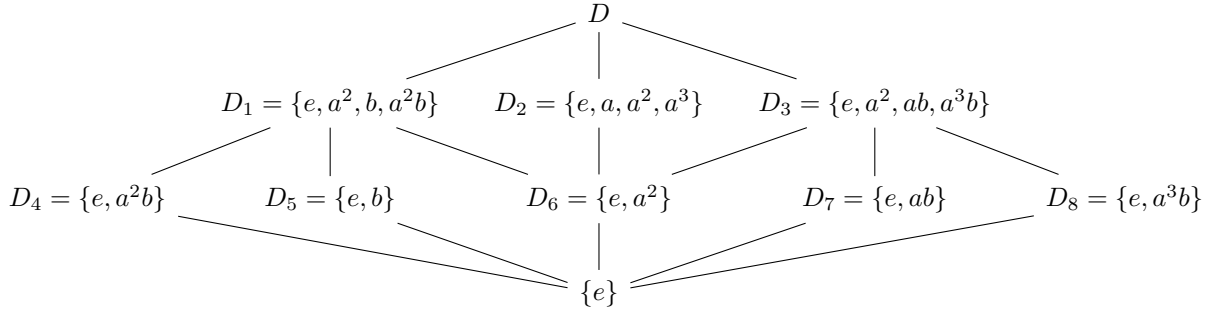
\[\begin{tikzcd}[column sep=small]
	&& {D} \\
	& {D_1=\{e,a^2,b,a^2b\}} & {D_2=\{e,a,a^2,a^3\}} & {D_3=\{e,a^2,ab,a^3b\}} \\
	{D_4=\{e,a^2b\}} & {D_5=\{e,b\}} & {D_6=\{e,a^2\}} & {D_7=\{e,ab\}} & {D_8=\{e,a^3b\}} \\
	&& {\{e\}}
	\arrow[no head, from=1-3, to=2-2]
	\arrow[no head, from=1-3, to=2-3]
	\arrow[no head, from=1-3, to=2-4]
	\arrow[no head, from=3-1, to=2-2]
	\arrow[no head, from=3-2, to=2-2]
	\arrow[no head, from=3-3, to=2-2]
	\arrow[no head, from=3-3, to=2-3]
	\arrow[no head, from=3-3, to=2-4]
	\arrow[no head, from=3-4, to=2-4]
	\arrow[no head, from=3-5, to=2-4]
	\arrow[no head, from=4-3, to=3-1]
	\arrow[no head, from=4-3, to=3-2]
	\arrow[no head, from=4-3, to=3-3]
	\arrow[no head, from=4-3, to=3-4]
	\arrow[no head, from=4-3, to=3-5]
\end{tikzcd}\]

Since $N\cong\zn^3$, $\wh{N}\cong \tn^3$. So, for any $\chi\in\wh{N}$, we may write $\chi=(u_1,u_2,u_3)$ for $u_1,u_2,u_3\in\tn$. For each $\chi\in\wh{N}$, let
\[\chi(ab^2a^{-1})=u_1,\quad\chi(b^{-2})=u_2,\quad\chi(c)=u_3.\]
Then $D\curvearrowright \wh{N}$ according to the following chart.
\[\begin{array}{c|c|c|c}
d & d\cdot u_1 & d\cdot u_2 & d\cdot u_3\\\hline
e & u_1 & u_2 & u_3\\
a & \ov{u_2} & u_1 & u_3\\
a^2 & \ov{u_1} & \ov{u_2} & u_3\\
a^3 & u_2 & \ov{u_1} & u_3\\
b & \ov{u_1} & u_2 & \ov{u_3}\\
ab & \ov{u_2} & \ov{u_1} & \ov{u_3}\\
a^2b & u_1 & \ov{u_2} & \ov{u_3}\\
a^3b & u_2 & u_1 & \ov{u_3}
\end{array}\]

Below we list the orbits and, within each orbit size, we have subdivided the characters into what we are terming {\em orbit types}. An orbit type simply means there is a single formula easily rendering all characters within that family. This is done for computational convenience.

\tbf{1-Orbits}: \label{example:one_orbits}
\[\begin{array}{c||c|c}
&\te{Type 1} & \te{Type 2} \\\hline
\te{Formula}&(1,1,\pm 1)&(-1,-1,\pm1)\\
\te{Stabilizer mod } N & D & D
\end{array}\]

\tbf{2-Orbits}: 
\[\begin{array}{c||c|c|c}
&\te{Type 1} & \te{Type 2} & \te{Type 3} \\\hline
\te{Formula}&(1,-1,\pm 1)&
(-1,1,\pm1)&(\ve,\ve,u_3)\\
\te{Stabilizer mod } N & D_1 & D_1 & D_2
\end{array}\]
where $\ve\in\{-1,1\}$ and $u_3\not\in\{-1,1\}$.

\tbf{4-Orbits}: 
\[\begin{array}{c||c|c|c|c|c}
&\te{Type 1} & \te{Type 2} & \te{Type 3} & \te{Type 4}&\te{Type 5}\\\hline
\te{Formula} &(u_1,u_1,\pm 1)&(u_1, \ov{u_1},\pm 1)&(u_1,\pm 1, \pm 1)&(\pm 1, u_2,\pm 1) &(\ve,-\ve,u_3)\\
\te{Stabilizer mod }N & D_8 & D_7 & D_4& D_5 & D_6
\end{array}\]
where $\ve\in\{-1,1\}$ and $u_1,u_2,u_3\not\in \{-1,1\}$.

\tbf{8-Orbits}: Any $\chi=(u_1,u_2,u_3)$ not meeting criteria of a previous category.

 We used the code to generate irreducible representations living over characters $\chi=(u_1,u_2,u_3)$ where $u_1,u_2,u_3\in\tn$ are specific, concrete values living in a particular orbit types. This means that we chose characters of the form $\chi=(i,i,-1)$ as opposed to $(v,v,-1)$ for $v$ any non-real element in $\tn$ as, currently, the code cannot manage indeterminate values for $u_j$. To determine the irreducible representations living over a particular orbit type, we manually determined formulas based on our numerical $\chi$. We then checked that the resulting irreducible representation was, in fact, a representation of $G$ and exhibited the appropriate behavior over $N$. Removing human intervention from the process is the focus of a future project. 
 
 The tables of irreducible representations arising from each orbit type are listed in Appendix \ref{appendix:start} where the notation $\pi^{(u_1,u_2,u_3)}_j$ represents the $j^{\te{th}}$ irreducible representation associated to $\chi=(u_1,u_2,u_3)$. Be aware that different presentations of $G$ and $N$ will yield different presentations of these irreducible representations (though they will be unitarily equivalent). GAP's method for determining presentations depends on a variety of methods and therefore results will vary from program to program. We also note that Corollary \ref{cor:k_path_connected_implies_cohomologous} means that our choice of orbit type ensure all characters in that orbit are cohomologous over their shared stabilizers. This is of particular importance as the dimension of $\pi_j^{(u_1,u_2,u_3)}$ is at least partially dependent on cohomology class.

\renewcommand\arraystretch{1}

We will use the notation $(u_1,u_2,u_3)\rightarrow (u_1',u_2',u_3')$ to represent a sequence of the same orbit type converging to a different orbit type. Additionally, let 
\[\stackrel[(u_1',u_2',u_3')]{}{L}(\pi_j^{(u_1,u_2,u_3)}):=\dlim{(u_1,u_2,u_3)}{(u_1',u_2',u_3')}\pi_j^{(u_1,u_2,u_3)}\]
be the matrix representation of $\pi_j^{(u_1,u_2,u_3)}$ with entrywise convergence as $(u_1,u_2,u_3)\rightarrow (u_1',u_2',u_3')$. We use $\stackrel[(u_1',u_2',u_3')]{}{BDL}(\pi_j^{(u_1,u_2,u_3)})$ to represent the block diagonalization of $\stackrel[(u_1',u_2',u_3')]{}{L}(\pi_j^{(u_1,u_2,u_3)})$ into irreducibles by some unitary $U_j$; that is,
\[U_j\l(\stackrel[(u_1',u_2',u_3')]{}{L}(\pi_j^{(u_1,u_2,u_3)})\r)U_j^{-1}=\stackrel[(u_1',u_2',u_3')]{}{BDL}(\pi_j^{(u_1,u_2,u_3)}).\]

To determine $\stackrel[(u_1',u_2',u_3')]{}{BDL}(\pi_j^{(u_1,u_2,u_3)})$, we used the character theory developed in Theorem \ref{thm:character_theory}. In the following examples, we show the resulting matrix of the irreducible representations over $(u_1,u_2,u_3)$ converging entrywise to a reducible representation and we include the direct sum to which this resulting matrix is equivalent and, finally, the unitary which effects this equivalence.

\begin{example}[8-orbit to 4-orbit (Type 4)] $(u_1,u_2,u_3)\rightarrow (-1,u_2,1)$ 
\[\begin{array}{c|c|c}
\stackrel[(-1,u_2,1)]{}{L}(\pi^{(u_1,u_2,u_3)}(a)) & \stackrel[(-1,u_2,1)]{}{L}(\pi^{(u_1,u_2,u_3)}(b)) & \stackrel[(-1,u_2,1)]{}{L}(\pi^{(u_1,u_2,u_3)}(c)) \\\hline\hline
&&\\
\l[\begin{array}{ccccccccc}
0 & 0 & 1 & 0 & 0 & 0 & 0 & 0\\
1 & 0 & 0 & 0 & 0 & 0 & 0 & 0\\
0 & 0 & 0 & 0 & 1 & 0 & 0 & 0\\
0 & 0 & 0 & 0 & 0 & 0 & 1 & 0\\
0 & 1 & 0 & 0 & 0 & 0 & 0 & 0\\
0 & 0 & 0 & 1 & 0 & 0 & 0 & 0\\
0 & 0 & 0 & 0 & 0 & 0 & 0 & 1\\
0 & 0 & 0 & 0 & 0 & 1 & 0 & 0
\end{array}\r]& \l[\begin{array}{cccccccc}
0 & 0 & 0 & \ov{u_2} & 0 & 0 & 0 &0\\
0 & 0 & 0 & 0 & 0 & 0 & 1 & 0\\
0 & 0 & 0 & 0 & 0 & -\ov{u_2}& 0 & 0\\
1 & 0 & 0 & 0 & 0 & 0 & 0 & 0 \\
0 & 0 & 0 & 0 & 0 & 0 & 0 & -1\\
0 & 0 & u_2 & 0 & 0 & 0 & 0 & 0\\
0 & -1 & 0 & 0 & 0 & 0 & 0 &0\\
0 & 0 & 0 & 0 & -u_2 & 0 & 0 & 0
\end{array}\r] & \l[\begin{array}{cccccccc}
1 & \\
 & 1 & \\
 & & 1 & \\
 & & & 1 & \\
 & & & & 1 & \\
 & & & & & 1\\
 & & & & & & 1\\
 & & & & & & & 1
\end{array}\r]
\end{array}\]
\[\stackrel[(-1,u_2,1)]{}{BDL}(\pi^{(u_1,u_2,u_3)})=\pi_1^{(-1,u_2,1)}\oplus \pi_2^{(-1,u_2,1)},\quad U=\dfrac{1}{\sqrt{2}}\l[\begin{array}{cccccccc}
1 & 0 & 0 & -\ov{u_2}^{1/2} & 0 & 0 & 0 & 0\\
0 & 0 & 1 & 0 & 0 & 0 & -\ov{u_2}^{1/2} & 0\\
0 & 1 & 0 & 0 & 0 & -\ov{u_2}^{1/2} & 0 & 0\\
0 & 0 & 0 & 0 & 1 & 0 & 0 & -\ov{u_2}^{1/2}\\
1 & 0 & 0 & -\ov{u_2}^{1/2} & 0 & 0 & 0 & 0\\
0 & 0 & 1 & 0 & 0 & 0 & \ov{u_2}^{1/2} & 0\\
0 & 1 & 0 & 0 & 0 & \ov{u_2}^{1/2} & 0 & 0\\
0 & 0 & 0 & 0 & 1 & 0 & 0 & \ov{u_2}^{1/2}
\end{array}\r]\]
\mbox{}
\end{example}

\begin{example}[8-orbit to 2-orbit (Type 3)] $(u_1,u_2,u_3)\rightarrow (1,1,u_3)$
\[\begin{array}{c|c|c}
\stackrel[(1,1,u_3)]{}{L}(\pi^{(u_1,u_2,u_3)}(a)) & \stackrel[(1,1,u_3)]{}{L}(\pi^{(u_1,u_2,u_3)}(b)) & \stackrel[(1,1,u_3)]{}{L}(\pi^{(u_1,u_2,u_3)}(c)) \\\hline\hline
&&\\
\l[\begin{array}{ccccccccc}
0 & 0 & 1 & 0 & 0 & 0 & 0 & 0\\
1 & 0 & 0 & 0 & 0 & 0 & 0 & 0\\
0 & 0 & 0 & 0 & 1 & 0 & 0 & 0\\
0 & 0 & 0 & 0 & 0 & 0 & 1 & 0\\
0 & 1 & 0 & 0 & 0 & 0 & 0 & 0\\
0 & 0 & 0 & 1 & 0 & 0 & 0 & 0\\
0 & 0 & 0 & 0 & 0 & 0 & 0 & 1\\
0 & 0 & 0 & 0 & 0 & 1 & 0 & 0
\end{array}\r]& \l[\begin{array}{cccccccc}
0 & 0 & 0 & 1 & 0 & 0 & 0 &0\\
0 & 0 & 0 & 0 & 0 & 0 & 1 & 0\\
0 & 0 & 0 & 0 & 0 & 1& 0 & 0\\
1 & 0 & 0 & 0 & 0 & 0 & 0 & 0 \\
0 & 0 & 0 & 0 & 0 & 0 & 0 & 1\\
0 & 0 & 1 & 0 & 0 & 0 & 0 & 0\\
0 & 1 & 0 & 0 & 0 & 0 & 0 &0\\
0 & 0 & 0 & 0 & 1 & 0 & 0 & 0
\end{array}\r] & \l[\begin{array}{cccccccc}
u_3 & \\
 & u_3 & \\
 & & u_3 & \\
 & & & \ov{u_3} & \\
 & & & & u_3 & \\
 & & & & & \ov{u_3}\\
 & & & & & & \ov{u_3}\\
 & & & & & & & \ov{u_3}
\end{array}\r]
\end{array}\]

\[\stackrel[(1,1,u_3)]{}{BDL}(\pi^{(u_1,u_2,u_3)})=\pi_1^{(1,1,u_3)}\oplus\pi_2^{(1,1,u_3)}\oplus \pi_3^{(1,1,u_3)}\oplus \pi_4^{(1,1,u_3)},\quad U=\dfrac{1}{2}\l[\begin{array}{cccccccc}
-1 & i & -i & 0 & 1 & 0 & 0 & 0\\
0 & 0 & 0 & -1 & 0 & -i & i & 1\\
-1 & -i & i & 0 & 1 & 0 & 0 & 0\\
0 & 0 & 0 & -1 & 0 & i & -i & 1 \\
1 & -1 & -1 & 0 & 1 & 0 & 0 & 0\\
0 & 0 & 0 & 1 & 0 & -1 & -1 & 1\\
1 & 1 & 1 & 0 & 1 & 0 & 0 & 0 \\
0 & 0 & 0 & 1 & 0 & 1 & 1 & 1
\end{array}\r]\]
\mbox{}

\end{example}

\begin{example}[8-orbit to 1-orbit (Type 1)]$(u_1,u_2,u_3)\rightarrow (1,1,1)$
\[\begin{array}{c|c|c}
\stackrel[(1,1,1)]{}{L}(\pi^{(u_1,u_2,u_3)}(a)) & \stackrel[(1,1,1)]{}{L}(\pi^{(u_1,u_2,u_3)}(b)) & \stackrel[(1,1,1)]{}{L}(\pi^{(u_1,u_2,u_3)}(c)) \\\hline\hline
&&\\
\l[\begin{array}{ccccccccc}
0 & 0 & 1 & 0 & 0 & 0 & 0 & 0\\
1 & 0 & 0 & 0 & 0 & 0 & 0 & 0\\
0 & 0 & 0 & 0 & 1 & 0 & 0 & 0\\
0 & 0 & 0 & 0 & 0 & 0 & 1 & 0\\
0 & 1 & 0 & 0 & 0 & 0 & 0 & 0\\
0 & 0 & 0 & 1 & 0 & 0 & 0 & 0\\
0 & 0 & 0 & 0 & 0 & 0 & 0 & 1\\
0 & 0 & 0 & 0 & 0 & 1 & 0 & 0
\end{array}\r]& \l[\begin{array}{cccccccc}
0 & 0 & 0 & 1 & 0 & 0 & 0 &0\\
0 & 0 & 0 & 0 & 0 & 0 & 1 & 0\\
0 & 0 & 0 & 0 & 0 & 1& 0 & 0\\
1 & 0 & 0 & 0 & 0 & 0 & 0 & 0 \\
0 & 0 & 0 & 0 & 0 & 0 & 0 & 1\\
0 & 0 & 1 & 0 & 0 & 0 & 0 & 0\\
0 & 1 & 0 & 0 & 0 & 0 & 0 &0\\
0 & 0 & 0 & 0 & 1 & 0 & 0 & 0
\end{array}\r] & \l[\begin{array}{cccccccc}
1 & \\
 & 1 & \\
 & & 1 & \\
 & & & 1 & \\
 & & & & 1 & \\
 & & & & & 1\\
 & & & & & & 1\\
 & & & & & & & 1
\end{array}\r]
\end{array}\]
\[\stackrel[(1,1,1)]{}{BDL}(\pi^{(u_1,u_2,u_3)})=\pi_1^{(1,1,1)}\oplus \pi_2^{(1,1,1)}\oplus \pi_3^{(1,1,1)}\oplus \pi_4^{(1,1,1)}\oplus \pi_5^{(1,1,1)}\oplus \pi_5^{(1,1,1)}\]
\[U=\dfrac{1}{2\sqrt{2}}\l[\begin{array}{rrrrrrrr}
1 & -1 & -1 & -1 & 1 & 1 & 1 & -1\\
-1 & 1 & 1 & -1 & -1 & 1 & 1 & -1 \\
1 & 1 & 1 & -1 & 1 & -1 & -1 & -1 \\
1 & 1 & 1 & 1 & 1 & 1 & 1 & 1\\
1 & 1 & -1 & 1 & -1 & -1 & 1 & -1\\
-1 & 1 & -1 & 1 & 1 & 1 & -1 & -1\\
1 & -1 & 1 & 1 & -1 & 1 & -1 & -1\\
1 & 1 & -1 & -1 & -1  & 1 & -1 & 1
\end{array}\r]\]
\end{example}

\renewcommand\arraystretch{2}

\begin{example}[4-orbit (Type 3) to 2-orbit (Type 1)] $(u_1,-1,1)\rightarrow(1,-1,1)$
\[\begin{array}{c|c|c|c}
& a & b & c\\\hline\hline
&&&\\[-.2in]
\stackrel[(1,-1,1)]{}{L}(\pi_1^{(u_1,-1,1)}) & {\renewcommand\arraystretch{1}  \l[\begin{array}{cccc}
0 & 1 & 0 & 0\\
0 & 0 & 0 & -1\\
1 & 0 & 0 & 0\\
0 & 0 & -1& 0
\end{array}\r]} & {\renewcommand\arraystretch{1} \l[\begin{array}{cccc}
0 & 0 & 0 & 1\\
0 & -1 & 0 & 0\\
0 & 0 & 1 & 0\\
-1 & 0 & 0 & 0
\end{array}\r]}&{\renewcommand\arraystretch{1} \l[\begin{array}{cccc}
1 & 0 & 0 & 0\\
0 & 1 & 0 & 0\\
0 & 0 & 1 & 0\\
0 & 0 & 0 & 1
\end{array}\r]}\\
&&&\\[-.25in]
\stackrel[(1,-1,1)]{}{L}(\pi_2^{(u_1,-1,1)})&{\renewcommand\arraystretch{1} \l[\begin{array}{cccc}
0 & 1 & 0 & 0\\
0 & 0 & 0 & 1\\
1 & 0 & 0 & 0\\
0 & 0 & 1& 0
\end{array}\r]} & {\renewcommand\arraystretch{1} \l[\begin{array}{cccc}
0 & 0 & 0 & 1\\
0 & 1 & 0 & 0\\
0 & 0 & -1 & 0\\
-1 & 0 & 0 & 0
\end{array}\r] }&{\renewcommand\arraystretch{1} \l[\begin{array}{cccc}
1 & 0 & 0 & 0\\
0 & 1 & 0 & 0\\
0 & 0 & 1 & 0\\
0 & 0 & 0 & 1
\end{array}\r]}
\end{array}\]
\[\stackrel[(1,-1,1)]{}{BDL}(\pi_1^{(u_1,-1,1)})=\pi^{(1,-1,1)},\quad U_1=\dfrac{1}{\sqrt{2}}{\renewcommand\arraystretch{1}\l[\begin{array}{cccc}
1 & 0 & 0 & 1\\
-i & 0 & 0 & i\\
0 & -1 & 1 & 0\\
0 & -i & -i & 0
\end{array}\r]}\]
\[\stackrel[(1,-1,1)]{}{BDL}(\pi_2^{(u_1,-1,1)})=\pi^{(1,-1,1)},\quad U_2=\dfrac{1}{\sqrt{2}}{\renewcommand\arraystretch{1}\l[\begin{array}{cccc}
1 & 0 & 0 & -1\\
i & 0 & 0 & i\\
0 & -1 & 1 & 0\\
0 & i & i & 0
\end{array}\r]}\]

\end{example}

\begin{example}[4-orbit (Type 3) to 1-orbit (Type 2)]\label{ex:4_orbit_to_1_orbit} $(u_1,-1,1)\rightarrow(-1,-1,1)$
\[\begin{array}{c|c|c|c}
& a & b & c\\\hline\hline
&&&\\[-.2in]
\stackrel[(-1,-1,1)]{}{L}(\pi_1^{(u_1,-1,1)}) & {\renewcommand\arraystretch{1}  \l[\begin{array}{cccc}
0 & 1 & 0 & 0\\
0 & 0 & 0 & -i\\
1 & 0 & 0 & 0\\
0 & 0 &i& 0
\end{array}\r]} & {\renewcommand\arraystretch{1} \l[\begin{array}{cccc}
0 & 0 & 0 & 1\\
0 & -i & 0 & 0\\
0 & 0 & -i & 0\\
-1 & 0 & 0 & 0
\end{array}\r]}&{\renewcommand\arraystretch{1} \l[\begin{array}{cccc}
1 & 0 & 0 & 0\\
0 & 1 & 0 & 0\\
0 & 0 & 1 & 0\\
0 & 0 & 0 & 1
\end{array}\r]}\\
&&&\\[-.25in]
\stackrel[(-1,-1,1)]{}{L}(\pi_2^{(u_1,-1,1)})&{\renewcommand\arraystretch{1} \l[\begin{array}{cccc}
0 & 1 & 0 & 0\\
0 & 0 & 0 & i\\
1 & 0 & 0 & 0\\
0 & 0 & -i& 0
\end{array}\r]} & {\renewcommand\arraystretch{1} \l[\begin{array}{cccc}
0 & 0 & 0 & 1\\
0 & i & 0 & 0\\
0 & 0 & i & 0\\
-1 & 0 & 0 & 0
\end{array}\r] }&{\renewcommand\arraystretch{1} \l[\begin{array}{cccc}
1 & 0 & 0 & 0\\
0 & 1 & 0 & 0\\
0 & 0 & 1 & 0\\
0 & 0 & 0 & 1
\end{array}\r]}
\end{array}\]
\[\stackrel[(-1,-1,1)]{}{BDL}(\pi_1^{(u_1,-1,1)})= \pi_3^{(-1,-1,1)}\oplus \pi_4^{(-1,-1,1)}\oplus \pi_5^{(-1,-1,1)},\quad U_1={\renewcommand\arraystretch{1}\l[\begin{array}{cccc}
1/2 & i/2 & -i/2 & i/2\\
1/2 & -i/2 & i/2 & i/2\\
0 & 1/\sqrt{2} & 1/\sqrt{2} & 0\\
1/\sqrt{2} & 0 & 0 & -i/\sqrt{2}
\end{array}\r]}\]
\[\stackrel[(-1,-1,1)]{}{BDL}(\pi_2^{(u_1,-1,1)}) = \pi_1^{(-1,-1,1)}\oplus \pi_2^{(-1,-1,1)}\oplus \pi_5^{(-1,-1,1)},\quad U_2={\renewcommand\arraystretch{1} \l[\begin{array}{cccc}
1/2 & -i/2 & i/2 & -i/2\\
i/2 & i/2 & -i/2 & i/2\\
1/\sqrt{2} & 0 & 0 & i/\sqrt{2}\\
0 & 1/\sqrt{2} & 1/\sqrt{2} & 0
\end{array}\r]}
\]
\end{example}

\renewcommand\arraystretch{2}
\begin{example}[2-orbit (Type 3) to 1-orbit (Type 2)] $(-1,-1,u_3)\rightarrow(-1,-1,1)$
\[\begin{array}{c|c|c|c}
& a & b & c\\\hline\hline
&&&\\[-.2in]
\stackrel[(-1,-1,1)]{}{L}(\pi_1^{(-1,-1,u_3)}) & {\renewcommand\arraystretch{1} \l[\begin{array}{cc}
-i & 0 \\
0 & -i
\end{array}\r]} & {\renewcommand\arraystretch{1}\l[\begin{array}{cc}
0 & -1 \\
1 & 0
\end{array}\r]}&{\renewcommand\arraystretch{1}\l[\begin{array}{cc}
1 & 0\\
0 & 1
\end{array}\r]}\\[\bigskipamount]
\stackrel[(-1,-1,1)]{}{L}(\pi_2^{(-1,-1,u_3)})&{\renewcommand\arraystretch{1} \l[\begin{array}{cc}
i & 0 \\
0 & i
\end{array}\r]} & {\renewcommand\arraystretch{1}\l[\begin{array}{cc}
0 & -1 \\
1 & 0
\end{array}\r]}&{\renewcommand\arraystretch{1}\l[\begin{array}{cc}
1 & 0\\
0 & 1
\end{array}\r]}\\[\bigskipamount]
\stackrel[(-1,-1,1)]{}{L}(\pi_3^{(-1,-1,u_3)})&{\renewcommand\arraystretch{1} \l[\begin{array}{cc}
-1 & 0 \\
0 & 1
\end{array}\r]} & {\renewcommand\arraystretch{1}\l[\begin{array}{cc}
0 & -1 \\
1 & 0
\end{array}\r]}&{\renewcommand\arraystretch{1}\l[\begin{array}{cc}
1 & 0\\
0 & 1
\end{array}\r]}\\[\bigskipamount]
\stackrel[(-1,-1,1)]{}{L}(\pi_4^{(-1,-1,u_3)})&{\renewcommand\arraystretch{1} \l[\begin{array}{cc}
1 & 0 \\
0 & -1
\end{array}\r]} & {\renewcommand\arraystretch{1}\l[\begin{array}{cc}
0 & -1 \\
1 & 0
\end{array}\r]}&{\renewcommand\arraystretch{1}\l[\begin{array}{cc}
1 & 0\\
0 & 1
\end{array}\r]}
\end{array}\]
\[\stackrel[(-1,-1,1)]{}{BDL}(\pi_1^{(-1,-1,u_3)})= \pi_2^{(-1,-1,1)}\oplus \pi_3^{(-1,-1,1)},\quad U_1={\renewcommand\arraystretch{1}\dfrac{1}{\sqrt{2}}\l[\begin{array}{cc}
1 & 1 \\
-i & i
\end{array}\r]}\]
\[\stackrel[(-1,-1,1)]{}{BDL}(\pi_2^{(-1,-1,u_3)})= \pi_1^{(-1,-1,1)}\oplus \pi_4^{(-1,-1,1)},\quad U_2= {\renewcommand\arraystretch{1}\dfrac{1}{\sqrt{2}}\l[\begin{array}{cc}
1 & 1 \\
-i & i
\end{array}\r]}\]
\[\stackrel[(-1,-1,1)]{}{BDL}(\pi_3^{(-1,-1,u_3)})= \pi_5^{(-1,-1,1)},\quad U_3= {\renewcommand\arraystretch{1}\dfrac{1}{\sqrt{2}}\l[\begin{array}{cc}
-1 & i\\
1 & i
\end{array}\r]}\]
\[\stackrel[(-1,-1,1)]{}{BDL}(\pi_4^{(-1,-1,u_3)})= \pi_5^{(-1,-1,1)},\quad U_4= {\renewcommand\arraystretch{1}\dfrac{1}{\sqrt{2}}\l[\begin{array}{cc}
1 & -i\\
1 & i
\end{array}\r]}\]
\end{example}

%%%%%%%%%%%%%%%%%%%%%%%%%%%%%%%%%%%%%%%%%%%%
%%%%%%%%%%%%%%%%%%%%%%%%%%%%%%%%%%%%%%%%%%%%

\newpage

\begin{appendices}

\section{Irreducible Representations of Dimension 3 Crystallographic Group 90}\label{appendix:start}

For each $\chi\in\wh{N}$, we let $\pi_j^{\chi}$ represent the $j^{\te{th}}$ element of $\wh{G}_{\chi}^{(\chi)}$. In particular, the number of inequivalent $\pi_j^{\chi}$ is exactly $|\wh{G}_{\chi}^{(\chi)}|$.

\subsection{Irreducible Representations Arising from 1-Orbits}

\renewcommand\arraystretch{1.85}

\[\begin{array}{c||c|c|c}
& a & b & c\\\hline\hline
\pi_1^{(1,1,1)} & [-1] & [-1] & [1] \\
\pi_2^{(1,1,1)} & [-1] & [1] & [1] \\
\pi_3^{(1,1,1)} & [1] & [-1] & [1]\\
\pi_4^{(1,1,1)} & [1] & [1] & [1]\\
\pi_5^{(1,1,1)} & {\renewcommand\arraystretch{1}\l[\begin{array}{cc}
0 & -1 \\
1 & 0
\end{array}\r]} & {\renewcommand\arraystretch{1} \l[\begin{array}{cc}
1 & 0 \\
0 & -1 
\end{array}\r]} & {\renewcommand\arraystretch{1} \l[\begin{array}{cc}
1 & 0 \\
0 & 1 
\end{array}\r]}\\[\bigskipamount]\hline
\pi_1^{(1,1,-1)} & [-1] & [-1] & [-1] \\
\pi_2^{(1,1,-1)} & [-1] & [1] & [-1] \\
\pi_3^{(1,1,-1)} & [1] & [-1] & [-1] \\
\pi_4^{(1,1,-1)} & [1] & [1] & [-1]\\
\pi_5^{(1,1,-1)} & {\renewcommand\arraystretch{1}\l[\begin{array}{cc}
0 & -1 \\
1 & 0
\end{array}\r]} & {\renewcommand\arraystretch{1}\l[\begin{array}{cc}
1 & 0 \\
0 & -1 
\end{array}\r]} & {\renewcommand\arraystretch{1}\l[\begin{array}{cc}
-1 & 0 \\
0 & -1 
\end{array}\r]}\\[\bigskipamount]\hline
\pi_1^{(-1,-1,1)} & [i] & [i] & [1] \\
\pi_2^{(-1,-1,1)} &[-i] & [i] & [1]\\
\pi_3^{(-1,-1,1)} & [-i] & [-i] & [1] \\
\pi_4^{(-1,-1,1)} & [i] & [-i] & [1] \\
\pi_5^{(-1,-1,1)} &{\renewcommand\arraystretch{1}\l[\begin{array}{cc}
0 & 1 \\
1 & 0
\end{array}\r]} & {\renewcommand\arraystretch{1}\l[\begin{array}{cc}
-i & 0 \\
0 & i
\end{array}\r]} & {\renewcommand\arraystretch{1}\l[\begin{array}{cc}
1 & 0 \\
0 & 1
\end{array}\r]}\\[\bigskipamount]\hline
\pi_1^{(-1,-1,-1)}&[i] & [i] & [-1] \\
\pi_2^{(-1,-1,-1)} &[-i] & [i] & [-1] \\
\pi_3^{(-1,-1,-1)}&[-i]&[-i]&[-1]\\
\pi_4^{(-1,-1,-1)}&[-i]&[-i]&[-1]\\
\pi_5^{(-1,-1,-1)} & {\renewcommand\arraystretch{1}\l[\begin{array}{cc}
0 & 1 \\
1 & 0
\end{array}\r]} & {\renewcommand\arraystretch{1}\l[\begin{array}{cc}
-i & 0 \\
0 & i 
\end{array}\r]} &{\renewcommand\arraystretch{1}\l[\begin{array}{cc}
-1 & 0 \\
0 & -1 
\end{array}\r]}
\end{array}\]

\newpage

\subsection{Irreducible Representations Arising from 2-Orbits}

\[\begin{array}{c||c|c|c}
& a & b & c\\\hline\hline
&&&\\[-.2in]
\pi^{(1,-1,1)} &{\renewcommand\arraystretch{1}\l[\begin{array}{cccc}
0 & 0 & -1 & 0 \\
0 & 0 & 0 & 1\\
1 & 0 & 0 & 0\\
0 & 1 & 0 & 0
\end{array}\r]} & {\renewcommand\arraystretch{1}\l[\begin{array}{cccc}
0 & -i & 0 & 0\\
-i & 0 & 0 & 0\\
0 & 0 & 0 & i\\
0 & 0 & -i & 0
\end{array}\r]} & {\renewcommand\arraystretch{1}\l[\begin{array}{cccc}
1 & 0 & 0 & 0\\
0 & 1 & 0 & 0\\
0 & 0 & 1 & 0\\
0 & 0 & 0 & 1
\end{array}\r]}\\[.3in]\hline
& & & \\[-.2in]
\pi^{(1,-1,-1)} &{\renewcommand\arraystretch{1}\l[\begin{array}{cccc}
0 & 0 & -1 & 0\\
0 & 0 & 0 & 1\\
1 & 0 & 0 & 0\\
0 & 1 & 0 & 0
\end{array}\r]} & {\renewcommand\arraystretch{1}\l[\begin{array}{cccc}
0 & -i & 0 & 0\\
-i & 0 & 0 & 0\\
0 & 0 & 0 & i\\
0 & 0 & -i & 0
\end{array}\r]} & {\renewcommand\arraystretch{1}\l[\begin{array}{cccc}
-1 & 0 & 0 & 0\\
0 & -1 & 0 & 0\\
0 & 0 & -1 & 0\\
0 & 0 & 0 & -1
\end{array}\r]}\\[.3in]\hline
& & & \\[-.2in]
\pi^{(-1,1,1)} &{\renewcommand\arraystretch{1}\l[\begin{array}{cccc}
0 & 0 & -1 & 0 \\
0 & 0 & 0 & 1\\
1 & 0 & 0 & 0\\
0 & 1 & 0 & 0
\end{array}\r]}&{\renewcommand\arraystretch{1}\l[\begin{array}{cccc}
0 & 1 & 0 & 0\\
1 & 0 & 0 & 0\\
0 & 0 & 0 & 1\\
0 & 0 & -1 & 0
\end{array}\r]} &{\renewcommand\arraystretch{1}\l[\begin{array}{cccc}
1 & 0 & 0 & 0\\
0 & 1 & 0 & 0\\
0 & 0 & 1 & 0\\
0 & 0 & 0 & 1
\end{array}\r]}\\[.3in]\hline
& & & \\[-.2in]
\pi^{(-1,1,-1)} &{\renewcommand\arraystretch{1}\l[\begin{array}{cccc}
0 & 0 & -1 & 0\\
0 & 0 & 0 & 1\\
1 & 0 & 0 & 0\\
0 & 1 & 0 & 0
\end{array}\r]}&{\renewcommand\arraystretch{1}\l[\begin{array}{cccc}
0 & 1 & 0 & 0\\
1 & 0 & 0 & 0\\
0 & 0 & 0 & 1\\
0 & 0 & -1 & 0
\end{array}\r]} &{\renewcommand\arraystretch{1}\l[\begin{array}{cccc}
-1 & 0 & 0 & 0\\
0 & -1 & 0 & 0\\
0 & 0 & -1 & 0\\
0 & 0 & 0 & -1
\end{array}\r]}\\[.3in]\hline
& & & \\[-.2in]
\pi_1^{(1,1,u_3)} & {\renewcommand\arraystretch{1} \l[\begin{array}{cc}
-i & 0 \\
0 & i
\end{array}\r]} & {\renewcommand\arraystretch{1}\l[\begin{array}{cc}
0 & 1 \\
1 & 0
\end{array}\r]}&{\renewcommand\arraystretch{1}\l[\begin{array}{cc}
u_3 & 0\\
0 & \ov{u_3}
\end{array}\r]}\\[\bigskipamount]
\pi_2^{(1,1,u_3)}&{\renewcommand\arraystretch{1} \l[\begin{array}{cc}
i & 0 \\
0 & -i
\end{array}\r]} & {\renewcommand\arraystretch{1}\l[\begin{array}{cc}
0 & 1 \\
1 & 0
\end{array}\r]}&{\renewcommand\arraystretch{1}\l[\begin{array}{cc}
u_3 & 0\\
0 & \ov{u_3}
\end{array}\r]}\\[\bigskipamount]
\pi_3^{(1,1,u_3)}&{\renewcommand\arraystretch{1} \l[\begin{array}{cc}
-1 & 0 \\
0 & -1
\end{array}\r]} & {\renewcommand\arraystretch{1}\l[\begin{array}{cc}
0 & 1 \\
1 & 0
\end{array}\r]}&{\renewcommand\arraystretch{1}\l[\begin{array}{cc}
u_3 & 0\\
0 & \ov{u_3}
\end{array}\r]}\\[\bigskipamount]
\makecell{\pi_4^{(1,1,u_3)}\\u_3\not\in\{-1,1\}}&{\renewcommand\arraystretch{1} \l[\begin{array}{cc}
1 & 0 \\
0 & 1
\end{array}\r]} & {\renewcommand\arraystretch{1}\l[\begin{array}{cc}
0 & 1 \\
1 & 0
\end{array}\r]}&{\renewcommand\arraystretch{1}\l[\begin{array}{cc}
u_3 & 0\\
0 & \ov{u_3}
\end{array}\r]}\\[.2in]\hline
& & & \\[-.2in]
\pi_1^{(-1,-1,u_3)} & {\renewcommand\arraystretch{1} \l[\begin{array}{cc}
-i & 0 \\
0 & -i
\end{array}\r]} & {\renewcommand\arraystretch{1}\l[\begin{array}{cc}
0 & -1 \\
1 & 0
\end{array}\r]}&{\renewcommand\arraystretch{1}\l[\begin{array}{cc}
u_3 & 0\\
0 & \ov{u_3}
\end{array}\r]}\\[\bigskipamount]
\pi_2^{(-1,-1,u_3)}&{\renewcommand\arraystretch{1} \l[\begin{array}{cc}
i & 0 \\
0 & i
\end{array}\r]} & {\renewcommand\arraystretch{1}\l[\begin{array}{cc}
0 & -1 \\
1 & 0
\end{array}\r]}&{\renewcommand\arraystretch{1}\l[\begin{array}{cc}
u_3 & 0\\
0 & \ov{u_3}
\end{array}\r]}\\[\bigskipamount]
\pi_3^{(-1,-1,u_3)}&{\renewcommand\arraystretch{1} \l[\begin{array}{cc}
-1 & 0 \\
0 & 1
\end{array}\r]} & {\renewcommand\arraystretch{1}\l[\begin{array}{cc}
0 & -1 \\
1 & 0
\end{array}\r]}&{\renewcommand\arraystretch{1}\l[\begin{array}{cc}
u_3 & 0\\
0 & \ov{u_3}
\end{array}\r]}\\[\bigskipamount]
\makecell{\pi_4^{(-1,-1,u_3)}\\u_3\not\in\{-1,1\}}&{\renewcommand\arraystretch{1} \l[\begin{array}{cc}
1 & 0 \\
0 & -1
\end{array}\r]} & {\renewcommand\arraystretch{1}\l[\begin{array}{cc}
0 & -1 \\
1 & 0
\end{array}\r]}&{\renewcommand\arraystretch{1}\l[\begin{array}{cc}
u_3 & 0\\
0 & \ov{u_3}
\end{array}\r]}
\end{array}\]

\newpage

\subsection{Irreducible Representations Arising from 4-Orbits}
\[\begin{array}{c||c|c|c}
& a & b & c\\\hline\hline
&&&\\[-.2in]
\pi_1^{(u_1,u_1,1)}&{\renewcommand\arraystretch{1}\l[\begin{array}{cccc}
0 & 1 & 0 & 0\\
0 & 0 & 0 & 1\\
1 & 0 & 0 & 0\\
0 & 0 & 1 & 0
\end{array}\r]} & {\renewcommand\arraystretch{1}\l[\begin{array}{cccc}
0 & 0 & -\ov{u_1} & 0\\
0 & 0 & 0 & -1\\
-1 & 0 & 0 & 0\\
0 & -u_1 & 0 & 0
\end{array}\r]} &{\renewcommand\arraystretch{1}\l[\begin{array}{cccc}
1 & 0 & 0 & 0\\
0 & 1 & 0 & 0\\
0 & 0 & 1 & 0\\
0 & 0 & 0 & 1
\end{array}\r]}\\[.35in]
\makecell{\pi_2^{(u_1,u_1,1)}\\u_1\not\in\{-1,1\}}&{\renewcommand\arraystretch{1}\l[\begin{array}{cccc}
0 & 1 & 0 & 0\\
0 & 0 & 0 & 1\\
1 & 0 & 0 & 0\\
0 & 0 & 1 & 0
\end{array}\r]} &  {\renewcommand\arraystretch{1}\l[\begin{array}{cccc}
0 & 0 & \ov{u_1} & 0\\
0 & 0 & 0 & 1\\
1 & 0 & 0 & 0\\
0 & u_1 & 0 & 0
\end{array}\r]} & {\renewcommand\arraystretch{1}\l[\begin{array}{cccc}
1 & 0 & 0 & 0\\
0 & 1 & 0 & 0\\
0 & 0 & 1 & 0\\
0 & 0 & 0 & 1
\end{array}\r]}\\[.3in]\hline
& & & \\[-.2in]
\pi_1^{(u_1,u_1,-1)}&{\renewcommand\arraystretch{1}\l[\begin{array}{cccc}
0 & 1 & 0 & 0\\
0 & 0 & 0 & 1\\
1 & 0 & 0 & 0\\
0 & 0 & 1 & 0
\end{array}\r]} & {\renewcommand\arraystretch{1}\l[\begin{array}{cccc}
0 & 0 & -\ov{u_1} & 0\\
0 & 0 & 0 & -1\\
-1 & 0 & 0 & 0\\
0 & -u_1 & 0 & 0
\end{array}\r]} &{\renewcommand\arraystretch{1}\l[\begin{array}{cccc}
-1 & 0 & 0 & 0\\
0 & -1 & 0 & 0\\
0 & 0 & -1 & 0\\
0 & 0 & 0 & -1
\end{array}\r]}\\[.35in]
\makecell{\pi_2^{(u_1,u_1,-1)}\\u_1\not\in\{-1,1\}}&{\renewcommand\arraystretch{1}\l[\begin{array}{cccc}
0 & 1 & 0 & 0\\
0 & 0 & 0 & 1\\
1 & 0 & 0 & 0\\
0 & 0 & 1 & 0
\end{array}\r]} &  {\renewcommand\arraystretch{1}\l[\begin{array}{cccc}
0 & 0 & \ov{u_1} & 0\\
0 & 0 & 0 & 1\\
1 & 0 & 0 & 0\\
0 & u_1 & 0 & 0
\end{array}\r]} & {\renewcommand\arraystretch{1}\l[\begin{array}{cccc}
-1 & 0 & 0 & 0\\
0 & -1 & 0 & 0\\
0 & 0 & -1 & 0\\
0 & 0 & 0 & -1
\end{array}\r]}\\[.3in]\hline
& & & \\[-.2in]
\pi_1^{(u_1,\ov{u_1},1)} & {\renewcommand\arraystretch{1}\l[\begin{array}{cccc}
0 & 1 & 0 & 0\\
0 & 0 & 0 & 1\\
1 & 0 & 0 & 0\\
0 & 0 & 1 & 0
\end{array}\r]} & {\renewcommand\arraystretch{1}\l[\begin{array}{cccc}
0 & 1 & 0 & 0\\
u_1 & 0 & 0 & 0\\
0 & 0 & 0 & \ov{u_1}\\
0 &  & 1 & 0
\end{array}\r]} & {\renewcommand\arraystretch{1}\l[\begin{array}{cccc}
1 & 0 & 0 & 0\\
0 & 1 & 0 & 0\\
0 & 0 & 1 & 0\\
0 & 0 & 0 & 1
\end{array}\r]}\\[.35in]
\makecell{\pi_2^{(u_1,\ov{u_1},1)}\\u_1\not\in\{-1,1\}}&{\renewcommand\arraystretch{1}\l[\begin{array}{cccc}
0 & 1 & 0 & 0\\
0 & 0 & 0 & 1\\
1 & 0 & 0 & 0\\
0 & 0 & 1 & 0
\end{array}\r]} & {\renewcommand\arraystretch{1}\l[\begin{array}{cccc}
0 & -1 & 0 & 0\\
-u_1 & 0 & 0 & 0\\
0 & 0 & 0 & -\ov{u_1}\\
0 & 0 & -1 & 0
\end{array}\r]} & {\renewcommand\arraystretch{1}\l[\begin{array}{cccc}
1 & 0 & 0 & 0\\
0 & 1 & 0 & 0\\
0 & 0 & 1 & 0\\
0 & 0 & 0 & 1
\end{array}\r]}\\[.3in]\hline
& & & \\[-.2in]
\pi_1^{(u_1,\ov{u_1},-1)} & {\renewcommand\arraystretch{1}\l[\begin{array}{cccc}
0 & 1 & 0 & 0\\
0 & 0 & 0 & 1\\
1 & 0 & 0 & 0\\
0 & 0 & 1 & 0
\end{array}\r]} & {\renewcommand\arraystretch{1}\l[\begin{array}{cccc}
0 & 1 & 0 & 0\\
u_1 & 0 & 0 & 0\\
0 & 0 & 0 & \ov{u_1}\\
0 & 0 & 1 & 0
\end{array}\r]} & {\renewcommand\arraystretch{1}\l[\begin{array}{cccc}
-1 & 0 & 0 & 0\\
0 & -1 & 0 & 0\\
0 & 0 & -1 & 0\\
0 & 0 & 0 & -1
\end{array}\r]}\\[.35in]
\makecell{\pi_2^{(u_1,\ov{u_1},-1)}\\u_1\not\in\{-1,1\}}&{\renewcommand\arraystretch{1}\l[\begin{array}{cccc}
0 & 1 & 0 & 0\\
0 & 0 & 0 & 1\\
1 & 0 & 0 & 0\\
0 & 0 & 1 & 0
\end{array}\r]} & {\renewcommand\arraystretch{1}\l[\begin{array}{cccc}
0 & -1 & 0 & 0\\
-u_1 & 0 & 0 & 0\\
0 & 0 & 0 & -\ov{u_1}\\
0 & 0 & -1 & 0
\end{array}\r]} & {\renewcommand\arraystretch{1}\l[\begin{array}{cccc}
-1 & 0 & 0 & 0\\
0 & -1 & 0 & 0\\
0 & 0 & -1 & 0\\
0 & 0 & 0 & -1
\end{array}\r]}\\[.3in]\hline
& & & \\[-.2in]
\pi_1^{(u_1,1,1)} & {\renewcommand\arraystretch{1}
\l[\begin{array}{cccc}
0 & 1 & 0 & 0\\
0 & 0 & 0 & u_1^{1/2}\\
1 & 0 & 0 & 0\\
0 & 0 & \ov{u_1}^{1/2}& 0
\end{array}\r]} & {\renewcommand\arraystretch{1}\l[\begin{array}{cccc}
0 & 0 & 0 & 1\\
0 & u_1^{1/2} & 0 & 0\\
0 & 0 & \ov{u_1}^{1/2} & 0\\
1 & 0 & 0 & 0
\end{array}\r]} & {\renewcommand\arraystretch{1}\l[\begin{array}{cccc}
1 & 0 & 0 & 0\\
0 & 1 & 0 & 0\\
0 & 0 & 1 & 0\\
0 & 0 & 0 & 1
\end{array}\r]}\\[.35in]
\makecell{\pi_2^{(u_1,1,1)}\\u_1\not\in\{-1,1\}}&{\renewcommand\arraystretch{1}\l[\begin{array}{cccc}
0 & 1 & 0 & 0\\
0 & 0 & 0 & -u_1^{1/2}\\
1 & 0 & 0 & 0\\
0 & 0 & -\ov{u_1}^{1/2}& 0
\end{array}\r]} & {\renewcommand\arraystretch{1}\l[\begin{array}{cccc}
0 & 0 & 0 & 1\\
0 & -u_1^{1/2} & 0 & 0\\
0 & 0 & -\ov{u_1}^{1/2} & 0\\
1 & 0 & 0 & 0
\end{array}\r]} & {\renewcommand\arraystretch{1}\l[\begin{array}{cccc}
1 & 0 & 0 & 0\\
0 & 1 & 0 & 0\\
0 & 0 & 1 & 0\\
0 & 0 & 0 & 1
\end{array}\r]}
\end{array}\]

\newpage

\[\begin{array}{c||c|c|c}
& a & b & c \\\hline\hline
&&&\\[-.2in]
\pi_1^{(u_1,1,-1)} & {\renewcommand\arraystretch{1}
\l[\begin{array}{cccc}
0 & 1 & 0 & 0\\
0 & 0 & 0 & u_1^{1/2}\\
1 & 0 & 0 & 0\\
0 & 0 & \ov{u_1}^{1/2}& 0
\end{array}\r]} & {\renewcommand\arraystretch{1}\l[\begin{array}{cccc}
0 & 0 & 0 & 1\\
0 & u_1^{1/2} & 0 & 0\\
0 & 0 & \ov{u_1}^{1/2} & 0\\
1 & 0 & 0 & 0
\end{array}\r]} & {\renewcommand\arraystretch{1}\l[\begin{array}{cccc}
-1 & 0 & 0 & 0\\
0 & -1 & 0 & 0\\
0 & 0 & -1 & 0\\
0 & 0 & 0 & -1
\end{array}\r]}\\[.35in]
\makecell{\pi_2^{(u_1,1,-1)}\\u_1\not\in\{-1,1\}}&{\renewcommand\arraystretch{1}\l[\begin{array}{cccc}
0 & 1 & 0 & 0\\
0 & 0 & 0 & -u_1^{1/2}\\
1 & 0 & 0 & 0\\
0 & 0 & -\ov{u_1}^{1/2}& 0
\end{array}\r]} & {\renewcommand\arraystretch{1}\l[\begin{array}{cccc}
0 & 0 & 0 & 1\\
0 & -u_1^{1/2} & 0 & 0\\
0 & 0 & -\ov{u_1}^{1/2} & 0\\
1 & 0 & 0 & 0
\end{array}\r]} & {\renewcommand\arraystretch{1}\l[\begin{array}{cccc}
-1 & 0 & 0 & 0\\
0 & -1 & 0 & 0\\
0 & 0 & -1 & 0\\
0 & 0 & 0 & -1
\end{array}\r]}\\[.3in]\hline
& & & \\[-.2in]
\pi_1^{(u_1,-1,1)} & {\renewcommand\arraystretch{1}
 \l[\begin{array}{cccc}
0 & 1 & 0 & 0\\
0 & 0 & 0 & -u_1^{1/2}\\
1 & 0 & 0 & 0\\
0 & 0 & -\ov{u_1}^{1/2}& 0
\end{array}\r]} & {\renewcommand\arraystretch{1} \l[\begin{array}{cccc}
0 & 0 & 0 & 1\\
0 & -u_1^{1/2} & 0 & 0\\
0 & 0 & \ov{u_1}^{1/2} & 0\\
-1 & 0 & 0 & 0
\end{array}\r] } & {\renewcommand\arraystretch{1}\l[\begin{array}{cccc}
1 & 0 & 0 & 0\\
0 & 1 & 0 & 0\\
0 & 0 & 1 & 0\\
0 & 0 & 0 & 1
\end{array}\r]}\\[.35in]
\makecell{\pi_2^{(u_1,-1,1)}\\u_1\not\in\{-1,1\}}&{\renewcommand\arraystretch{1} \l[\begin{array}{cccc}
0 & 1 & 0 & 0\\
0 & 0 & 0 & u_1^{1/2}\\
1 & 0 & 0 & 0\\
0 & 0 & \ov{u_1}^{1/2}& 0
\end{array}\r]} & {\renewcommand\arraystretch{1}\l[\begin{array}{cccc}
0 & 0 & 0 & 1\\
0 & u_1^{1/2} & 0 & 0\\
0 & 0 & -\ov{u_1}^{1/2} & 0\\
-1 & 0 & 0 & 0
\end{array}\r] } & {\renewcommand\arraystretch{1}\l[\begin{array}{cccc}
1 & 0 & 0 & 0\\
0 & 1 & 0 & 0\\
0 & 0 & 1 & 0\\
0 & 0 & 0 & 1
\end{array}\r]}\\[.3in]\hline
& & & \\[-.2in]
\pi_1^{(u_1,-1,-1)} & {\renewcommand\arraystretch{1}
 \l[\begin{array}{cccc}
0 & 1 & 0 & 0\\
0 & 0 & 0 & -u_1^{1/2}\\
1 & 0 & 0 & 0\\
0 & 0 & -\ov{u_1}^{1/2}& 0
\end{array}\r]} & {\renewcommand\arraystretch{1} \l[\begin{array}{cccc}
0 & 0 & 0 & 1\\
0 & -u_1^{1/2} & 0 & 0\\
0 & 0 & \ov{u_1}^{1/2} & 0\\
-1 & 0 & 0 & 0
\end{array}\r] } & {\renewcommand\arraystretch{1}\l[\begin{array}{cccc}
-1 & 0 & 0 & 0\\
0 & -1 & 0 & 0\\
0 & 0 & -1 & 0\\
0 & 0 & 0 & -1
\end{array}\r]}\\[.35in]
\makecell{\pi_2^{(u_1,-1,-1)}\\u_1\not\in\{-1,1\}}&{\renewcommand\arraystretch{1} \l[\begin{array}{cccc}
0 & 1 & 0 & 0\\
0 & 0 & 0 & u_1^{1/2}\\
1 & 0 & 0 & 0\\
0 & 0 & \ov{u_1}^{1/2}& 0
\end{array}\r]} & {\renewcommand\arraystretch{1}\l[\begin{array}{cccc}
0 & 0 & 0 & 1\\
0 & u_1^{1/2} & 0 & 0\\
0 & 0 & -\ov{u_1}^{1/2} & 0\\
-1 & 0 & 0 & 0
\end{array}\r] }& {\renewcommand\arraystretch{1}\l[\begin{array}{cccc}
-1 & 0 & 0 & 0\\
0 & -1 & 0 & 0\\
0 & 0 & -1 & 0\\
0 & 0 & 0 & -1
\end{array}\r]}\\[.3in]\hline
& & & \\[-.2in]
\pi_1^{(1,u_2,1)} & {\renewcommand\arraystretch{1}\l[\begin{array}{cccc}
0 & 1 & 0 & 0\\
0 & 0 & 0 & 1\\
1 & 0 & 0 & 0\\
0 & 0 & 1 & 0
\end{array}\r]} & {\renewcommand\arraystretch{1}\l[\begin{array}{cccc}
-\ov{u_2}^{1/2} & 0 & 0 & 0\\
0 & 0 & -\ov{u_2}^{1/2} & 0\\
0 & -u_2^{1/2} & 0 & 0\\
0 & 0 & 0 & -u_2^{1/2}
\end{array}\r]} & {\renewcommand\arraystretch{1}\l[\begin{array}{cccc}
1 & 0 & 0 & 0\\
0 & 1 & 0 & 0\\
0 & 0 & 1 & 0\\
0 & 0 & 0 & 1
\end{array}\r]}\\[.35in]
\makecell{\pi_2^{(1,u_2,1)}\\u_2\not\in\{-1,1\}} & {\renewcommand\arraystretch{1}\l[\begin{array}{cccc}
0 & 1 & 0 & 0\\
0 & 0 & 0 & 1\\
1 & 0 & 0 & 0\\
0 & 0 & 1 & 0
\end{array}\r]} & {\renewcommand\arraystretch{1}\l[\begin{array}{cccc}
\ov{u_2}^{1/2} & 0 & 0 & 0\\
0 & 0 & \ov{u_2}^{1/2} & 0\\
0 & u_2^{1/2} & 0 & 0\\
0 & 0 & 0 & u_2^{1/2}
\end{array}\r]} & {\renewcommand\arraystretch{1}\l[\begin{array}{cccc}
1 & 0 & 0 & 0\\
0 & 1 & 0 & 0\\
0 & 0 & 1 & 0\\
0 & 0 & 0 & 1
\end{array}\r]}\\[.3in]\hline
& & & \\[-.2in]
\pi_1^{(1,u_2,-1)} & {\renewcommand\arraystretch{1}\l[\begin{array}{cccc}
0 & 1 & 0 & 0\\
0 & 0 & 0 & 1\\
1 & 0 & 0 & 0\\
0 & 0 & 1 & 0
\end{array}\r]} & {\renewcommand\arraystretch{1}\l[\begin{array}{cccc}
-\ov{u_2}^{1/2} & 0 & 0 & 0\\
0 & 0 & -\ov{u_2}^{1/2} & 0\\
0 & -u_2^{1/2} & 0 & 0\\
0 & 0 & 0 & -u_2^{1/2}
\end{array}\r]} & {\renewcommand\arraystretch{1}\l[\begin{array}{cccc}
-1 & 0 & 0 & 0\\
0 & -1 & 0 & 0\\
0 & 0 & -1 & 0\\
0 & 0 & 0 & -1
\end{array}\r]}\\[.35in]
\makecell{\pi_2^{(1,u_2,-1)}\\u_2\not\in\{-1,1\}} & {\renewcommand\arraystretch{1}\l[\begin{array}{cccc}
0 & 1 & 0 & 0\\
0 & 0 & 0 & 1\\
1 & 0 & 0 & 0\\
0 & 0 & 1 & 0
\end{array}\r]} & {\renewcommand\arraystretch{1}\l[\begin{array}{cccc}
\ov{u_2}^{1/2} & 0 & 0 & 0\\
0 & 0 & \ov{u_2}^{1/2} & 0\\
0 & u_2^{1/2} & 0 & 0\\
0 & 0 & 0 & u_2^{1/2}
\end{array}\r]} & {\renewcommand\arraystretch{1}\l[\begin{array}{cccc}
-1 & 0 & 0 & 0\\
0 & -1 & 0 & 0\\
0 & 0 & -1 & 0\\
0 & 0 & 0 & -1
\end{array}\r]}
\end{array}\]

\newpage

\[\begin{array}{c||c|c|c}
& a & b & c \\\hline\hline
&&&\\[-.2in]
\pi_1^{(-1,u_2,1)} & {\renewcommand\arraystretch{1}\l[\begin{array}{cccc}
0 & 1 & 0 & 0\\
0 & 0 & 0 & 1\\
1 & 0 & 0 & 0\\
0 & 0 & 1 & 0
\end{array}\r]} & {\renewcommand\arraystretch{1}\l[\begin{array}{cccc}
-\ov{u_2}^{1/2} & 0 & 0 & 0\\
0 & 0 & \ov{u_2}^{1/2} & 0\\
0 & -u_2^{1/2} & 0 & 0\\
0 & 0 & 0 & u_2^{1/2}
\end{array}\r]} & {\renewcommand\arraystretch{1}\l[\begin{array}{cccc}
1 & 0 & 0 & 0\\
0 & 1 & 0 & 0\\
0 & 0 & 1 & 0\\
0 & 0 & 0 & 1
\end{array}\r]}\\[.35in]
\makecell{\pi_2^{(-1,u_2,1)}\\u_2\not\in\{-1,1\}} & {\renewcommand\arraystretch{1}\l[\begin{array}{cccc}
0 & 1 & 0 & 0\\
0 & 0 & 0 & 1\\
1 & 0 & 0 & 0\\
0 & 0 & 1 & 0
\end{array}\r]} & {\renewcommand\arraystretch{1}\l[\begin{array}{cccc}
\ov{u_2}^{1/2} & 0 & 0 & 0\\
0 & 0 & -\ov{u_2}^{1/2} & 0\\
0 & u_2^{1/2} & 0 & 0\\
0 & 0 & 0 & -u_2^{1/2}
\end{array}\r]} & {\renewcommand\arraystretch{1}\l[\begin{array}{cccc}
1 & 0 & 0 & 0\\
0 & 1 & 0 & 0\\
0 & 0 & 1 & 0\\
0 & 0 & 0 & 1
\end{array}\r]}\\[.3in]\hline
& & & \\[-.2in]
\pi_1^{(-1,u_2,-1)} & {\renewcommand\arraystretch{1}\l[\begin{array}{cccc}
0 & 1 & 0 & 0\\
0 & 0 & 0 & 1\\
1 & 0 & 0 & 0\\
0 & 0 & 1 & 0
\end{array}\r]} & {\renewcommand\arraystretch{1}\l[\begin{array}{cccc}
-\ov{u_2}^{1/2} & 0 & 0 & 0\\
0 & 0 & \ov{u_2}^{1/2} & 0\\
0 & -u_2^{1/2} & 0 & 0\\
0 & 0 & 0 & u_2^{1/2}
\end{array}\r]} & {\renewcommand\arraystretch{1}\l[\begin{array}{cccc}
-1 & 0 & 0 & 0\\
0 & -1 & 0 & 0\\
0 & 0 & -1 & 0\\
0 & 0 & 0 & -1
\end{array}\r]}\\[.35in]
\makecell{\pi_2^{(-1,u_2,-1)}\\u_2\not\in\{-1,1\}} & {\renewcommand\arraystretch{1}\l[\begin{array}{cccc}
0 & 1 & 0 & 0\\
0 & 0 & 0 & 1\\
1 & 0 & 0 & 0\\
0 & 0 & 1 & 0
\end{array}\r]} & {\renewcommand\arraystretch{1}\l[\begin{array}{cccc}
\ov{u_2}^{1/2} & 0 & 0 & 0\\
0 & 0 & -\ov{u_2}^{1/2} & 0\\
0 & u_2^{1/2} & 0 & 0\\
0 & 0 & 0 & -u_2^{1/2}
\end{array}\r]} & {\renewcommand\arraystretch{1}\l[\begin{array}{cccc}
-1 & 0 & 0 & 0\\
0 & -1 & 0 & 0\\
0 & 0 & -1 & 0\\
0 & 0 & 0 & -1
\end{array}\r]}\\[.3in]\hline
& & & \\[-.2in]
\pi_1^{(1,-1,u_3)} & {\renewcommand\arraystretch{1}\l[\begin{array}{cccc}
0 & -1 & 0 & 0\\
1 & 0 & 0 & 0 \\
0 & 0 & 0 & 1 \\
0 & 0 & 1 & 0
\end{array}\r]} & {\renewcommand\arraystretch{1}\l[\begin{array}{cccc}
0 & 0 & -1 & 0\\
0 & 0 & 0 & 1\\
1 & 0 & 0 & 0\\
0 & 1 & 0 & 0
\end{array}\r]} & {\renewcommand\arraystretch{1}\l[\begin{array}{cccc}
u_3 & 0 & 0 & 0\\
0 & u_3 & 0 & 0\\
0 & 0 & \ov{u_3} & 0\\
0 & 0 & 0 & \ov{u_3}
\end{array}\r]}\\[.35in]
\makecell{\pi_2^{(1,-1,u_3)}\\u_3\not\in\{-1,1\}} & {\renewcommand\arraystretch{1}\l[\begin{array}{cccc}
0 & 1 & 0 & 0\\
1 & 0 & 0 & 0 \\
0 & 0 & 0 & -1 \\
0 & 0 & 1 & 0
\end{array}\r]} & {\renewcommand\arraystretch{1}\l[\begin{array}{cccc}
0 & 0 & -1 & 0\\
0 & 0 & 0 & -1\\
1 & 0 & 0 & 0\\
0 & -1 & 0 & 0
\end{array}\r]} & {\renewcommand\arraystretch{1}\l[\begin{array}{cccc}
u_3 & 0 & 0 & 0\\
0 & u_3 & 0 & 0\\
0 & 0 & \ov{u_3} & 0\\
0 & 0 & 0 & \ov{u_3}
\end{array}\r]}\\[.3in]\hline
& & & \\[-.2in]
\pi_1^{(-1,1,u_3)} & {\renewcommand\arraystretch{1}\l[\begin{array}{cccc}
0 & -1 & 0 & 0\\
1 & 0 & 0 & 0 \\
0 & 0 & 0 & 1 \\
0 & 0 & 1 & 0
\end{array}\r]} & {\renewcommand\arraystretch{1}\l[\begin{array}{cccc}
0 & 0 & 1 & 0\\
0 & 0 & 0 & 1\\
1 & 0 & 0 & 0\\
0 & -1 & 0 & 0
\end{array}\r]} & {\renewcommand\arraystretch{1}\l[\begin{array}{cccc}
u_3 & 0 & 0 & 0\\
0 & u_3 & 0 & 0\\
0 & 0 & \ov{u_3} & 0\\
0 & 0 & 0 & \ov{u_3}
\end{array}\r]}\\[.35in]
\makecell{\pi_2^{(-1,1,u_3)}\\u_3\not\in\{-1,1\}} & {\renewcommand\arraystretch{1}\l[\begin{array}{cccc}
0 & 1 & 0 & 0\\
1 & 0 & 0 & 0 \\
0 & 0 & 0 & -1 \\
0 & 0 & 1 & 0
\end{array}\r]} & {\renewcommand\arraystretch{1}\l[\begin{array}{cccc}
0 & 0 & 1 & 0\\
0 & 0 & 0 & -1\\
1 & 0 & 0 & 0\\
0 & 1 & 0 & 0
\end{array}\r]} & {\renewcommand\arraystretch{1}\l[\begin{array}{cccc}
u_3 & 0 & 0 & 0\\
0 & u_3 & 0 & 0\\
0 & 0 & \ov{u_3} & 0\\
0 & 0 & 0 & \ov{u_3}
\end{array}\r]}
\end{array}\]

\subsection{Irreducible Representations Arising from 8-Orbits}

$(u_1,u_2,u_3)$ not of any previous types.

\renewcommand\arraystretch{1}

\[\pi^{(u_1,u_2,u_3)}(a)=\l[\begin{array}{ccccccccc}
0 & 0 & 1 & 0 & 0 & 0 & 0 & 0\\
1 & 0 & 0 & 0 & 0 & 0 & 0 & 0\\
0 & 0 & 0 & 0 & 1 & 0 & 0 & 0\\
0 & 0 & 0 & 0 & 0 & 0 & 1 & 0\\
0 & 1 & 0 & 0 & 0 & 0 & 0 & 0\\
0 & 0 & 0 & 1 & 0 & 0 & 0 & 0\\
0 & 0 & 0 & 0 & 0 & 0 & 0 & 1\\
0 & 0 & 0 & 0 & 0 & 1 & 0 & 0
\end{array}\r]\]
\vspace{.05in}

\[\pi^{(u_1,u_2,u_3)}(b)=\l[\begin{array}{cccccccc}
0 & 0 & 0 & \ov{u_2} & 0 & 0 & 0 &0\\
0 & 0 & 0 & 0 & 0 & 0 & 1 & 0\\
0 & 0 & 0 & 0 & 0 & u_1\ov{u_2}& 0 & 0\\
1 & 0 & 0 & 0 & 0 & 0 & 0 & 0 \\
0 & 0 & 0 & 0 & 0 & 0 & 0 & u_1\\
0 & 0 & u_2 & 0 & 0 & 0 & 0 & 0\\
0 & \ov{u_1} & 0 & 0 & 0 & 0 & 0 &0\\
0 & 0 & 0 & 0 & \ov{u_1}u_2 & 0 & 0 & 0
\end{array}\r]\]
\vspace{.05in}

\[\pi^{(u_1,u_2,u_3)}(c)=\l[\begin{array}{cccccccc}
u_3 & \\
 & u_3 & \\
 & & u_3 & \\
 & & & \ov{u_3} & \\
 & & & & u_3 & \\
 & & & & & \ov{u_3}\\
 & & & & & & \ov{u_3}\\
 & & & & & & & \ov{u_3}
\end{array}\r]\]

\end{appendices}

\printbibliography

\end{document}